\documentclass[11pt,a4paper,reqno]{amsart}
 
\usepackage[normalem]{ulem}

\usepackage[utf8]{inputenc}

\usepackage[top=2.7cm,bottom=2.7cm,left=2.4cm,right=2.4cm]{geometry}

\usepackage{amsfonts,amsmath,amssymb,amsthm,dsfont,amsxtra,enumerate,multicol,mathtools,tkz-berge,hyperref,graphicx}
\usepackage{enumitem}
\usepackage{listings}
\usepackage{xcolor}
\usepackage[ruled,vlined]{algorithm2e}

\definecolor{codegreen}{rgb}{0,0.6,0}
\definecolor{codegray}{rgb}{0.5,0.5,0.5}
\definecolor{codepurple}{rgb}{0.58,0,0.82}
\definecolor{backcolour}{rgb}{0.95,0.96,0.98}

\lstdefinestyle{mystyle}{
    backgroundcolor=\color{backcolour},   
    commentstyle=\color{codegreen},
    keywordstyle=\color{magenta},
    numberstyle=\tiny\color{codegray},
    stringstyle=\color{codepurple},
    basicstyle=\ttfamily\footnotesize,
    breakatwhitespace=false,         
    breaklines=true,                 
    captionpos=b,                    
    keepspaces=true,                 
    numbers=left,                    
    numbersep=5pt,                  
    showspaces=false,                
    showstringspaces=false,
    showtabs=false,                  
    tabsize=2
}

\newtheorem{theorem}{Theorem}[section]
\newtheorem{proposition}[theorem]{Proposition}
\newtheorem{lemma}[theorem]{Lemma}
\newtheorem{corollary}[theorem]{Corollary}

\theoremstyle{definition}

\newtheorem{remark}[theorem]{Remark}
\newtheorem{example}[theorem]{Example}
\newtheorem{notation}[theorem]{Notation}
\newtheorem{question}[theorem]{Question}

\newtheorem{construction}[theorem]{Construction}

\DeclareMathOperator{\reg}{reg}

\DeclareMathOperator{\N}{N}
\DeclareMathOperator{\lcm}{lcm}
\DeclareMathOperator{\supp}{supp}
\DeclareMathOperator{\level}{level}

\newcommand{\T}{\mathcal{T}}
\newcommand{\Sc}{\mathcal{S}}
\newcommand{\cN}{\mathcal{N}}

\newcommand{\G}{\mathcal{G}}

\newcommand{\lex}{\text{lex}}
\author[M. Amalore Nambi]{Marie Amalore Nambi}
\address{Sabanci University, Faculty of Engineering and Natural Sciences, Orta Mahalle, Tuzla, 34956, Istanbul, Turkey}
\email{amalore.p@gmail.com, amalore.pushparaj@sabanciuniv.edu}

\author[A. A. Qureshi]{Ayesha Asloob Qureshi}
\email{aqureshi@sabanciuniv.edu, ayesha.asloob@sabanciuniv.edu}

\title{Squarefree powers of closed neighborhood ideals}

\date{}

\begin{document}


\thanks{The authors are supported by the Scientific and Technological Research Council of Turkey T\"UB\.{I}TAK under the Grant No: 124F113, and thankful to T\"UB\.{I}TAK for their supports. }

\date{}

\begin{abstract}
    In this article, we characterize all trees whose highest non-vanishing squarefree power of the closed neighborhood ideal is componentwise linear. In addition, we investigate the Castelnuovo-Mumford regularity of the $\nu$-th squarefree power of the closed neighborhood ideal of trees and show that this number can be arbitrarily larger than the degree of the ideal. Finally, we give a formula for the regularity of $\nu$-th squarefree power of the closed neighborhood ideal of caterpillar graphs.
\end{abstract}

\maketitle

\section{Introduction}

The core of combinatorial commutative algebra lies in the study of the interplay between the algebraic properties of an ideal and the combinatorial structure of the underlying object. In particular, squarefree monomial ideals are naturally connected with combinatorial objects such as simplicial complexes, graphs, and clutters, offering the study of their algebraic and homological properties via the associated combinatorial data. The existence of a linear resolution has been one of the most extensively studied homological invariants. Recall that a \textit{linear resolution} is a minimal free resolution in which each differential map is represented by matrix that have entries in the set of linear forms. These ideals are well-behaved and computationally simpler to work with than most others. A remarkable result in this direction is due to Fr\"{o}berg \cite{F1990} who establishes that the edge ideal of a graph admits a linear resolution if and only if the graph is co-chordal. 

\vspace{2mm}

Herzog and Hibi \cite{HH1999} introduced the notion of componentwise linear ideals, which form the next best class after ideals with linear resolutions. Let $I \subseteq S=K[x_1,\ldots,x_n]$ be a graded ideal. Recall that $I$ is said to be \textit{componentwise linear} if, for each degree $k$, the ideal generated by all homogeneous elements of $I$ of degree $k$  has a linear resolution. It is known from \cite{HH1999} that a Stanley-Reisner ideal $I$ is sequentially Cohen-Macaulay if and only if its Alexander dual is componentwise linear. For further details on componentwise linearity, we refer the reader to the survey article \cite{HV2022}. In this context, we study the componentwise linearity of the highest non-vanishing squarefree power of the closed neighborhood ideal of trees is of particular importance.

\vspace{2mm}

Sharifan and Moradi \cite{SM2020} introduced the \textit{closed neighborhood ideal} $NI(G)$ of a finite simple graph $G$, a class of squarefree monomial ideals generated by monomials corresponding to the closed neighborhoods of the vertices of $G$. This ideal can be viewed as the facet ideal of the neighborhood complex $\cN(G)$ of $G$. From the perspective of the Stanley–Reisner complex, its Alexander dual is identified with the dominance complex. In recent years, several algebraic properties of closed neighborhood ideals such as regularity, projective dimension, Cohen-Macaulayness,   minimal irreducible decompositions, normally torsion-free property, and persistence property, have been studied in \cite{CJRS2025,HS2022, KNQ, NQBM2022, NQ2024, NBR2025, SM2020, S2025}.  To the best of our knowledge, this is the first study that focuses on squarefree power of these ideals in the literature.

\vspace{2mm}

Let $I$ be a monomial ideal. The $k$-th squarefree power of $I$, denoted by $I^{[k]}$, is the ideal generated by all squarefree monomials belonging to $I^k$. From a combinatorial perspective, the squarefree powers of $I$ are closely related to the matching theory of simplicial complexes or clutters. Specifically, the minimal generators of $I^{[k]}$ are $k$-matchings of the associated simplicial complex. The squarefree powers of an ideal carry important information about its ordinary powers. In particular, if $I^{[k]}$ does not have a componentwise linear, then $I^k$ does not have a componentwise linear (see \cite[Lemma 4.4]{HHZ2004}). Bigdeli, Herzog, and Zaare-Nahandi \cite[Theorem 5.1]{BHZ2018} showed that the highest non-vanishing squarefree power of the edge ideal of a graph has linear quotients. Kamberi, Navarra, and the second author of this article \cite{KNQ} showed that the highest non-vanishing squarefree power of the facet ideal of a simplicial tree with the intersection property has linear quotients. However, they further demonstrated that this is not true in general for simplicial trees. A natural question that arises in this context is the following.
\begin{question} \label{qus.CWL}
    Is the highest non-vanishing squarefree power of the closed neighborhood ideal of a tree is componentwise linear?
\end{question}
However, Question \ref{qus.CWL} does not hold in general as observed in Example \ref{ex.counter}. Motivated by this, we provide a combinatorial characterization of trees whose highest non-vanishing squarefree power of the closed neighborhood ideal is componentwise linear (see Theorems \ref{thm.1}, \ref{thm.2}, and \ref{thm.3}). 

\vspace{2mm}

The second part of this article is dedicated to study the Castelnuovo-Mumford regularity  (regularity) of the highest non-vanishing squarefree power of the closed neighborhood ideal of graphs. Sharifan and Moradi in \cite{SM2020} investigated the regularity of the closed neighborhood ideal of various classes of graphs. In particular, they obtained the exact formula for the regularity of the closed neighborhood ideal of path graphs, generalized star graphs, complete $r$-partite graphs and $m$-book graphs. Moreover, for a forest, they obtained a lower bound for the regularity of the closed neighborhood ideal in terms of the matching number. Chakraborty et al. \cite{CJRS2025} obtained a precise formula for the regularity of the closed neighborhood ideal of forests in terms of the matching number. Furthermore, they showed that, for any graph $G$, the matching number of $G$ provides a lower bound for $\reg(S/NI(G))$. A natural question is whether the difference between the regularity of the highest non-vanishing squarefree power of the closed neighborhood ideal of a graph and the degree of the ideal can become arbitrarily large. More precisely, we may ask the following question.

\begin{question} \label{qus.regm}
    Given an integer $m$. Is there exists a connected graph such that $\reg(S/(NI(G)^{[\nu]}))-\deg(NI(G)^{[\nu]})=m$?
\end{question}
We give a positive answer to Question \ref{qus.regm} (see Theorems \ref{thm.regm} and \ref{thm.regm2}). This indicates that determining the regularity of the highest non-vanishing squarefree power of the closed neighborhood ideal for an arbitrary tree is a hard problem. Therefore, we conclude by focusing on a specific class of trees, namely caterpillar graphs (see Theorem \ref{thm.reg}).

\vspace{2mm}


\section{Required Ingredients} \label{sec.P}

We first recall some basic definitions, notations, and results from graph theory and commutative algebra that will be used throughout the subsequent sections.

\subsection*{Graph Theory} Let $G$ be a finite simple graph with vertex set $V(G)$ and edge set $E(G)$. For a subset $U \subseteq V(G)$, the \textit{induced subgraph} of $G$ on $U$, denoted by $G[U]$, is the graph with vertex set $U$ and edge set $E(G[U])=\{e\in E(G) : e \subseteq U\}$.

A sequence of vertices $v_1, \ldots, v_k$ in $G$ is said to be a \textit{path} of length $k-1$ if $\{v_i,v_{i+1}\} \in E(G)$ for all $ 1\leq i \leq k-1$. Given any $u,v\in V(G)$, we set the length of the shortest path connecting $u$ and $v$ by $d(u,v)$. Moreover, a sequence of vertices $v_1, \ldots, v_k$ in $G$ is said to be a \textit{cycle} of length $k$ if $\{v_i,v_{i+1}\} \in E(G)$ for all $ 1\leq i \leq k$ where $v_1=v_{k+1}$. A graph $G$ is called a \textit{forest} if it does not contain a cycle as an induced subgraph. A {\it tree} is a connected forest. 

For a vertex $u \in V(G)$, let $G\setminus u$ denote the induced subgraph of $G$ on the vertex set $V(G) \setminus \{u\}$. For a vertex $u \in V(G)$, the set of neighbors of $u$, denoted by $N_G(u)$, is the set $\{v \in V(G) : \{u,v\} \in E(G)\}$. The degree of the vertex $u$ in $G$ is the cardinality of $N_G(u)$, and is denoted by $\deg_G(u)$. The set of closed neighbors of $u$, denoted by $N_G[u]$, is the set $N_G(u) \cup \{u\}$.

\subsection*{Neighborhood Complex} A simplicial complex $\Delta$ on the vertex set $V(\Delta)=[n]$ is a non-empty collection of subsets of $V(\Delta)$ such that if $F \in \Delta$ and $G \subseteq F$, then $G \in \Delta$. The elements of $\Delta$ are called faces. For any $F \in \Delta$, the dimension of $F$ is one less than the cardinality of $F$. The maximal faces of $\Delta$ with respect to inclusion are called facets, and the set of all facets of $\Delta$ is denoted by $\mathcal{F}(\Delta)$. A subcollection of a simplicial complex $\Delta$ is a simplicial complex whose facet set is a subset of the facet set of $\Delta$.

A facet $F$ of a simplicial complex $\Delta$ is called a leaf if either $F$ is the only facet of $\Delta$ or there exists a facet $G \in \Delta$ such that $G\neq F$ and $F \cap G \supseteq F \cap H$ for all facets $H \neq F$ of $\Delta$. A simplicial complex $\Delta$ is a {\it simplicial forest} if every nonempty subcollection of $\Delta$ has a leaf. A connected simplicial forest is called a {\it simplicial tree}.

The closed neighborhood complex of a graph $G$, denoted by $\mathcal{N}(G)$, is the simplicial complex on vertex set $V(G)$ defined by 
\begin{equation*}
\begin{split}
\mathcal{N}(G)=&\{F \subseteq V(G) : \text{there exists vertex } u \text{ such that } F \subseteq N_G[u] 
\text{ and } N_G[v] \not\subseteq F \\& \text{ for all } v \in V(G) \setminus u
\}.
\end{split}
\end{equation*}

Set $[n]=\{1,\ldots,n\}$. Below, we present an example of the closed neighborhood complex of a graph.

\begin{example} \label{ex.facet}
    Let $G$ be a graph on $[10]$ with edge set $$E(G)=\{\{1,2\},\{2,3\},\{3,4\},\{4,5\},\{5,6\},\{6,7\},\{7,8\},\{8,9\},\{5,10\}\}.$$ Then the facets of the closed neighborhood complex $\cN(G)$ are $$\mathcal{F}(\cN(G))=\{\{1,2\},\{5,10\},\{8,9\},\{2,3,4\},\{3,4,5\},\{5,6,7\},\{6,7,8\}\}.$$
\end{example}

As shown in \cite{B1989}, the following remark recalls that the closed neighborhood complex of a tree is a simplicial forest. 

\begin{remark} \cite[Chapter 5, Example 3, page 174]{B1989} \cite[Theorem 3.2]{HHTZ2008}
Let $G$ be a tree. Then the closed neighborhood complex $\mathcal{N}(G)$ is totally balanced. Equivalently, $\mathcal{N}(G)$ is a simplicial forest.
\end{remark}

\subsection*{Squarefree powers and matchings} Let $S = K[x_1,\ldots, x_n]$ be a polynomial ring in $n$ variables over a field $K$. Let $I$ be a monomial ideal of $S$. Let $\mathcal{G}(I)$ denote the minimal monomial generating set of a monomial ideal $I$. 

We denote by $I^{[a]}$ the ideal generated by the squarefree monomials belonging to $I^a$, by $I_{\langle a \rangle}$ the ideal generated by all homogeneous elements of degree $a$ belonging to $I$, and by $\deg(I)$ the largest degree among the minimal generators of $I$.

For a subset $A\subseteq [n]$ we associate a squarefree monomial ${\bf x}_A := \prod_{u\in A} x_u \in S$, and the set $A$ is called the support of ${\bf x}_A$, denoted by $\supp({\bf x}_A)$. Let $\Delta$ be a simplicial complex. The facet ideal of $\Delta$ is defined as $I(\Delta)=({\bf x}_F : F\in \mathcal{F}(\Delta)) \subset S$. The facet ideal of the closed neighborhood complex of a graph $G$ is called the closed neighborhood ideal of $G$, denoted by $NI(G)$.

\begin{example}
    Let $G$ be the graph as in Example \ref{ex.facet}, and let $\cN(G)$ be its closed neighborhood complex. Then the closed neighborhood ideal of $G$ is
    $$NI(G)= (x_1x_2,x_5x_{10},x_8x_9,x_2x_3x_4,x_3x_4x_5,x_5x_6x_7,x_6x_7x_8) \subset S=K[x_1,\ldots,x_{10}].$$
\end{example}

A \textit{matching} of $\Delta$ is a set of pairwise disjoint facets of $\Delta$. For a matching $M$ of $\Delta$, we set 
\[
V_M=\{u\in V(\Delta) : u \in F \text{ and } F \in  M\}.
\]
A matching consisting of $k$ facets is referred to as a $k$-matching. A $k$-matching is called maximal, if $\Delta$ does not admit any $(k+1)$-matching. The \textit{matching number} of $\Delta$ is the size of a maximal
matching of $\Delta$ and is denoted by $\nu(\Delta)$. Let $G$ be a graph and $M$ be a $k$-matching of $\cN(G)$. For a subgraph $H$ of $G$, we denote by 
\[
M_{H}=\{F\in M : F \subseteq V(H)\},
\]
the subcollection of facets of $M$ whose vertex sets are contained in $V(H)$.

\vspace{2mm}

We demonstrate the above definitions with the following example.

\begin{example}
    Let $G$ be the graph described in Example \ref{ex.facet}, and let $\cN(G)$ denote its closed neighborhood complex. Then $M=\{\{1,2\},\{5,10\},\{8,9\}\}$ is a maximal matching of $\cN(G)$, and hence $\nu(\cN(G))=3$. Let $H=G\setminus \{5\}$ be the induced subgraph of $G$. Then ${M}_{H}=\{\{1,2\},\{8,9\}\}$.
\end{example}

\subsection{Componentwise Linearity and Linear Relations}
Let $S = K[x_1,\ldots, x_n]$ be a polynomial ring over a field $K$ and $I$ be a homogeneous ideal of $S$. Let $\mathbf{F}_{\bullet}$ be a minimal graded $S$-free resolution of $I$:
$$\mathbf{F}_{\bullet}:  \oplus_{j \in \mathbb{Z}}S(-j)^{\beta_{p,j}(I)} \stackrel{\phi_{p}} \longrightarrow  \oplus_{j \in \mathbb{Z}}S(-j)^{\beta_{p-1,j}(I)} \longrightarrow  \cdots \longrightarrow  \oplus_{j \in \mathbb{Z}}S(-j)^{\beta_{0,j}(I)} \stackrel{\phi_{0}} \longrightarrow I\longrightarrow  0,$$ where $p \leq n$. The numbers $\beta_{i,j}(I)$ are called the $(i,j)$-th graded Betti numbers of $I$. 
The \emph{Castelnuovo-Mumford regularity} or simply  \emph{regularity} of $I$ over $S$, denoted by $\reg(I)$, is defined as
$$\reg(I) \coloneqq \max \{j : \beta_{i,i+j}(I) \neq 0, \text{ for some } i\}.$$
Moreover, we have $\beta_{i,j}(I)=\beta_{i+1,j}(S/I)$ for all $i\geq 0$ and $j \in \mathbb{Z}$.

\vspace{2mm}

The following lemma is a key ingredient that will be used repeatedly in Section \ref{sec.regularity}.
\begin{lemma}\cite[Lemma 3.1]{HT2010} \label{Lemma1.Reg}
Let $I$ be a homogeneous ideal of $S$ and $f$ be a degree $d$ monomial of $S$. Consider $ 0 \rightarrow S/(I:f)  (-d)\rightarrow S/(I) \rightarrow S/(I+f)  \rightarrow  0$ the short exact sequence. Then one has
     $$\reg(S/(I+f)) \leq \max \{\reg(S/(I:f))+d-1, \reg(S/I)\}.$$
     The equality holds if $\reg(S/(I:f))+d \neq \reg(S/I)$.
\end{lemma}

A homogeneous ideal $I\subset S$ has a $d$-\textit{linear resolution} if $\beta_{i,j}(I)=0$ for all $i \geq 0$ and for $j \neq i+d$. A homogeneous ideal $I\subset S$ has \textit{linear quotients} if there is a system of minimal homogeneous generators $u_1, u_2,\ldots,u_n$ such that the colon ideal $(u_1,\dots,u_{i-1}):u_i$ is generated by linear forms for each $i = 2,\ldots,n$. From \cite[Proposition 8.2.1]{HH2011}, it follows that if $I$ has linear quotients and generated in degree $d$ then $I$ has a $d$-linear resolution. A homogeneous ideal $I\subset S$ is said to be \textit{componentwise linear} if $I_{\langle j \rangle}$ has a linear resolution for all $j$. From \cite[Theorem 8.2.15]{HH2011}, it follows that if $I$ has linear quotients then $I$ is componentwise linear. An ideal $I\subset S$, generated by homogeneous elements of degree $d$, is said to be \textit{linearly related} if $\beta_{1,j}(I) = 0$ for all $j \neq 1 + d$. In particular, if $I$ is not linearly related then $I$ is not componentwise linear.

\vspace{2mm}

It is known from \cite[Theorem 5.1]{BHZ2018} that the $\nu$-th squarefree power of the edge ideal of a graph has linear quotients, and hence is componentwise linear. However, this result does not hold in general for monomial ideals. In the following example, we show that the $\nu$-th squarefree power of the closed neighborhood ideal of a graph does not need to be componentwise linear. 

\begin{example} \label{ex.counter}
    Let $G$ be the graph given in Example \ref{ex.facet}, and let $\cN(G)$ denote its closed neighborhood complex. Then $\nu(\cN(G))=3$. Macaulay2 \cite{GS} computation shows that $(NI(G)^{[3]})_{\langle 7 \rangle}$ does not have a linear resolution, therefore the ideal $NI(G)^{[3]}$ is not componentwise linear. 
\end{example}

In the following, we recall a useful tool from \cite{BHZ2018} to investigate the linearly related property of monomial ideals.

\vspace{2mm}

 Let $I$ be a monomial ideal generated in degree $d$. In \cite{BHZ2018}, the authors associated a graph $G_I$ to $I$ as follows: $V(G_I)=\G(I)$, and  $\{u,v\} \in E(G_I)$ if and only if $\deg(\mathrm{lcm}(u,v))=d+1$. Moreover, for all $u,v\in \G(I)$, the induced subgraph of $G_I$ on the vertex set $\{w\in V(G_I): w\text{ divides } \lcm(u,v)\}$ is denoted by $G^{(u,v)}_{I}$. 
  
		\begin{theorem}\cite[Corollary 2.2]{BHZ2018}\label{Rem.LP}
		   Let $I$ be a monomial ideal generated in degree d. Then $I$ is linearly related if and only if for all $u,v\in \G(I)$ there is a path in $G^{(u,v)}_{I}$ connecting $u$ and $v$. 
		\end{theorem}


\section{Some lemmas on matching of simplicial trees} \label{sec.MG}

In this section, we establish several combinatorial properties of $k$-matchings of simplicial trees, which will be used in the subsequent section.

\begin{construction}\label{cons:bipartite}
    Let $\Delta$ be a simplicial complex, and let $M$ and $N$ be two distinct $k$-matchings of $\Delta$. Then there exist $\hat{M}\subset M$ and $\hat{N}\subset N$ such that $\hat{M}\cap \hat{N}=\emptyset$ and $M\setminus \hat{M}= N\setminus \hat{N}$. That is, $\hat{M}$ and $\hat{N}$ are obtained by removing the common elements of $M$ and $N$. Then $\hat{M}$ and $\hat{N}$ form a $\hat{k}$-matching of $\Delta$ with $\hat{k}\leq k$. Let $\hat{M}=\{F_1, \ldots, F_{\hat{k}}\}$ and $\hat{N}=\{G_1, \ldots, G_{\hat{k}}\}$. We construct the bipartite graph $B_{(M,N)}$ on the vertex set $\hat{M} \sqcup \hat{N}$ and the edge set $\{ \{F_i,G_j\} : F_i \cap G_j \neq \emptyset\}$.
\end{construction}

The following remark follows from the construction above.

\begin{remark}\label{rem=cons}
     Following the notation in Construction~\ref{cons:bipartite}, if $F_i \subset V_{\hat{N}}=\bigcup_{j=1}^{\hat{k}}G_j$, then the degree of $F_i$ in $B_{(M,N)}$ is at least $2$. 
\end{remark}

We next show that the graph $B_{(M,N)}$ associated to two $k$-matchings of a simplicial tree has a particularly simple structure.

\begin{lemma} \label{lem.tree}
    Let $\Delta$ be a simplicial tree, and let $M$ and $N$ be two distinct $k$-matchings of $\Delta$. Then the bipartite graph $B_{(M,N)}$ defined in Construction~\ref{cons:bipartite} is acyclic. 
\end{lemma}
\begin{proof}
Suppose, to the contrary, that $B_{(M,N)}$ has a cycle 
\[
\mathcal{C}= F_{1},G_{1},F_{2}, \ldots,F_{q},G_{q},F_{1}
\]
of minimal length $q\geq 4$, after reordering the vertices if necessary.
Consider the subcomplex $\Delta'=\langle F_{1},\ldots,F_{q}, G_{1},\ldots, G_{q} \rangle\subseteq \Delta$. Then $\Delta'$ is also a simplicial forest with a leaf, say $G_{1}$. It follows from the construction of $B_{(M,N)}$ that $G_{1} \cap F_{1}$ and $G_{1} \cap F_{2}$ are non-empty, and by the minimality of $q$, $G_1\cap F_i=\emptyset$ for all $3\leq i\leq q$. Since $M$ is a matching, $F_1$ and $F_2$ are disjoint, and therefore $G_{1} \cap F_{1}$ and $G_{1} \cap F_{2}$ are also disjoint. This implies that $G_{1}$ is not a leaf of $\Delta'$, a contradiction to $\Delta$ being a simplicial tree.
\end{proof}

\begin{proposition} \label{Prop2.m=n}
    Let $\Delta$ be a simplicial tree, and let $M$ and $N$ be two $k$-matchings of $\Delta$. Suppose that $V_{M} \subseteq V_{N}$. Then $M=N$.
\end{proposition}
\begin{proof}
Following Construction~\ref{cons:bipartite}, consider the bipartite graph $B_{(M,N)}$. Since $V_{M} \subseteq V_{N}$, by the construction of $\hat{M}$ and $\hat{N}$ it follows that $V_{\hat{M}} \subseteq V_{\hat{N}}$. In particular, for each $F_i \in \hat{M}$, we have $F_i \subseteq V_{\hat{N}}$. Then, following Remark~\ref{rem=cons}, the degree of $F_i$ is at least $2$ for all $1\leq i\leq \hat{k}$. Consequently, the number of edges of $B_{(M,N)}$ is at least $2\hat{k}$. 

On the other hand, since $B_{(M,N)}$ is acyclic by Lemma~\ref{lem.tree}, an acyclic graph on $2\hat{k}$ vertices has at most $2\hat{k}-1$ edges. This is a contradiction. Hence $M=N$, as desired.
\end{proof}

As an immediate consequence, we obtain the following description of the elements of $\mathcal{G}(NI(\T)^{[\nu]})$, where $\T$ is a tree. 

\begin{corollary} \label{cor.unique}
    Let $\T$ be a tree, and let $\mathcal{N}(\T)$ be the closed neighborhood complex with $\mathcal{F}(\mathcal{N}(\T))=\{F_1,\ldots,F_n\}$. Then each $u \in \mathcal{G}(NI(\T)^{[k]})$ can be written uniquely as 
    \[
    u={\bf x}_{F_{i_1}}\cdots {\bf x}_{F_{i_k}},
    \]
    where $M=\{F_{i_1},\ldots, F_{i_k}\}$ is a $k$-matching of $\mathcal{N}(\T)$.
\end{corollary}

We recall a useful result from \cite{KNQ} related to the matchings of a simplicial tree. 

\begin{theorem} \cite[Lemma 3.5]{KNQ} \label{thm.intersection}
     Let $\Delta$ be a simplicial tree. Further, let $M =\{F_1,\ldots, F_k\}$ and $N = \{G_1,\ldots,G_k\}$ be two $k$-matchings of $\Delta$. Then there exist $i, j \in \{1,\ldots, k\}$ such that $F_i \cap G_s = \emptyset$ for all  $s \neq  j$.
\end{theorem}

\begin{remark} \label{Rem.intersection}
    Moreover, by removing the common facets from $M$ and $N$ in Theorem \ref{thm.intersection}, one may conclude that  there exist $i, j \in \{1,\ldots, k\}$ such that $F_i\neq G_j$ and $F_i \cap G_s = \emptyset$ for all  $s \neq  j$.  Moreover, if $k$ is the matching number of $\Delta$, then  $F_i\cap G_j\neq\emptyset$.
\end{remark}

The following lemmas describe combinatorial properties of two $\nu$-matchings of $\cN(\T)$, and they play a key role in the proof of the Theorem~\ref{thm.3}. In particular, we compare the numbers of facets contained in the components of $\T\setminus\{u\}$. 

\begin{lemma} \label{lem.car}
Let $\T$ be a tree and let $u$ be a vertex of $\T$. Let $\T_1,\ldots,\T_d$ be the connected components of the graph $\T\setminus\{u\}$. Suppose that $M$ and $N$ are two $\nu$-matchings of $\mathcal{N}(\T)$. Then, for each $i=1,\ldots,d$, one has
\[
|M_{\T_i}|-1 \;\leq\; |N_{\T_i}| \;\leq\; |M_{\T_i}|+1.
\]
\end{lemma}

\begin{proof}
By symmetry, it suffices to prove the inequalities for $i=1$. First, suppose that $a=|N_{\T_1}|<|M_{\T_1}|-1$ . The set $M_{\T_1} \,\cup\, \bigcup_{i=2}^d N_{\T_i}$ forms a matching of $\mathcal{N}(\T)$, since all its facets are contained in
pairwise disjoint components of $\T\setminus\{u\}$. Moreover, $\N_{\T}[u]$ is the unique facet of $\mathcal{N}(\T)$ that intersects
more than one component of $\T\setminus\{u\}$. Hence, among the facets of $N$,
at most one intersects more than one component, and therefore $\sum_{i=2}^d |N_{\T_i}|\geq \nu-(a+1).$ It follows that
\[
\bigl|M_{\T_1} \cup \bigcup_{i=2}^d N_{\T_i}\bigr|
= |M_{\T_1}| + \nu - a - 1 > \nu,
\]
which contradicts the fact that $\nu$ is the matching number of
$\mathcal{N}(\T)$. Therefore, $|N_{\T_1}|\ge |M_{\T_1}|-1.$ Similarly, by interchanging the roles of $M$ and $N$, we obtain $|M_{\T_1}|\ge |N_{\T_1}|-1$ which gives $|M_{\T_1}|+1\ge |N_{\T_1}|.$
\end{proof}

\begin{lemma}\label{lemma1.redu}
Under the assumptions of Lemma~\ref{lem.car}, additionally suppose that
$\N_{\T}[u]\in M$. Then the following hold.
\begin{enumerate}

\item[{\em(i)}]\label{rem.c11}
If $\N_{\T}[u]$ does not intersect any facet of $N_{\T_i}$ for some $i$, then $|M_{\T_i}|=|N_{\T_i}|$.
\item[{\em(ii)}]\label{rem.c2}
If $\N_{\T}[u]\notin N$ and $u\in V_N$, then either $|M_{\T_i}|=|N_{\T_i}|$ for all $i$, or there exist distinct indices $i,j$ such that $|N_{\T_i}|=|M_{\T_i}|+1$, $|N_{\T_j}|=|M_{\T_j}|-1$, and
$|N_{\T_k}|=|M_{\T_k}|$ for all $k\neq i,j$.

\item[{\em(iii)}]\label{rem.c3}
If $u\notin V_N$, then there exists an index $i$ such that $N_{\T_i}$ contains a facet that contains a neighbor of $u$, and
\[
|N_{\T_i}|=|M_{\T_i}|+1,
\]
while $|N_{\T_j}|=|M_{\T_j}|$ for all $j\neq i$.
\end{enumerate}
\end{lemma}

\begin{proof}
Since $\N_{\T}[u]\in M$ and every other facet of $M$ belongs to exactly one
$M_{\T_i}$, we have
\begin{equation}\label{eq=M}
M=\{\N_{\T}[u]\}\sqcup M_{\T_1}\sqcup\cdots\sqcup M_{\T_d},
\qquad
\nu=1+|M_{\T_1}|+\cdots+|M_{\T_d}|.
\end{equation}

\noindent
(i) 
Fix $i$. If no facet of $N_{\T_i}$ intersects $\N_{\T}[u]$, then
$(M\setminus M_{\T_i})\cup N_{\T_i}$ is a matching of $\mathcal{N}(\T)$, hence it has
cardinality at most $\nu$. Since $|M\setminus M_{\T_i}|=\nu-|M_{\T_i}|$, this gives
$|N_{\T_i}|\le |M_{\T_i}|$. Similarly, $(N\setminus N_{\T_i})\cup M_{\T_i}$ is a matching,
so $|M_{\T_i}|\le |N_{\T_i}|$. Therefore $|M_{\T_i}|=|N_{\T_i}|$.

\medskip
\noindent
(ii) If $\N_{\T}[u]\notin N$ and $u\in V_N$, then there exists a vertex $v$ adjacent
to $u$ such that $\N_{\T}[v]\in N$, and
\begin{equation}\label{eq=N}
N=\{\N_{\T}[v]\}\sqcup N_{\T_1}\sqcup\cdots\sqcup N_{\T_d},
\qquad
\nu=1+|N_{\T_1}|+\cdots+|N_{\T_d}|.
\end{equation}
We may assume that $v\in V(\T_1)$. For each $i=2,\ldots,d$, the facets in
$M_{\T_i}$ are disjoint from $\N_{\T}[v]$. Hence, if $|N_{\T_i}|=|M_{\T_i}|-1$ for some $i\neq1$, then replacing the facets of $N_{\T_i}$ in $N$ by those of $M_{\T_i}$ produces a matching of $\mathcal{N}(\T)$ of size $\nu+1$, a contradiction. On the other hand, if
$|N_{\T_1}|=|M_{\T_1}|+1$, then replacing the facets of $M_{\T_1}$ in $M$ by those of $N_{\T_1}$ yields a matching of size $\nu+1$, since all facets of $N_{\T_1}$ are disjoint from $\N_{\T}[u]$, again a contradiction. Therefore, by Lemma~\ref{lem.car},
\[
|M_{\T_1}|-1\le |N_{\T_1}|\le |M_{\T_1}|,
\qquad
|M_{\T_i}|\le |N_{\T_i}|\le |M_{\T_i}|+1
\ \text{for } i=2,\ldots,d.
\]
Using (\ref{eq=M}) and (\ref{eq=N}), this implies that either $|N_{\T_i}|=|M_{\T_i}|$ for all $i$, or  we have $|N_{\T_1}|=|M_{\T_1}|-1$, $|N_{\T_j}|=|M_{\T_j}|+1$ for some $j>1$, and $|N_{\T_k}|=|M_{\T_k}|$ for all $k\neq 1,j$. In addition, it follows from (i) that $N_{\T_j}$ intersects $\N_{\T}[u]$.

\medskip
\noindent
(iii) If $u\notin V_N$, then
\begin{equation}\label{eq=N1}
N=N_{\T_1}\sqcup\cdots\sqcup N_{\T_d},
\qquad
\nu=|N_{\T_1}|+\cdots+|N_{\T_d}|.
\end{equation}

Suppose that $|N_{\T_i}|=|M_{\T_i}|-1$ for some $i$. Replacing the facets of $N_{\T_i}$ in $N$ by those of $M_{\T_i}$ produces a matching of $\mathcal{N}(\T)$ of size $\nu+1$, again a contradiction. Therefore, by
Lemma~\ref{lem.car}, we have $|M_{\T_i}|\le |N_{\T_i}|\le |M_{\T_i}|+1$ for all $i$. Finally, using (\ref{eq=M}) and (\ref{eq=N1}), it follows that there exists exactly one index $i$ such that $|N_{\T_i}|=|M_{\T_i}|+1$, and from (i) we have that $N_{\T_i}$ intersects $\N_{\T}[u]$.
\end{proof}

\begin{lemma}\label{lemma.redu}
Let $\T$ be a tree and let $M$ and $N$ be two $\nu$-matchings of $\mathcal{N}(\T)$. Let $F\in M$ and suppose that
\[
|\{G\in N : F\cap G \neq \emptyset\}|=s >2.
\]
Then there exists a $\nu$-matching $W$ of
$\mathcal{N}(\T)$ different from $M$ and $N$ such that $W\subset M\cup N$ with
\[
|\{H\in W : F\cap H \neq \emptyset\}|= 2.
\]
\end{lemma}

\begin{proof}
Let $F=\N_{\T}[u]\in M$. Then $\deg_{\T}(u)=d\ge s$. Since $s> 3$, we have 
$\N_{\T}[u]\notin N$. Let $\T_1,\ldots,\T_d$ be the connected components of 
$\T\setminus\{u\}$.

By Lemma~\ref{lemma1.redu}(i)--(iii), the inequality  $|N_{\T_i}|\neq |M_{\T_i}|$ can occur for at most two indices $i$, and only for those $i$ such that $N_{\T_i}$ contains a facet 
intersecting $F$. Since $s>2$, there exist at least two indices $i\neq j$ such that 
$N_{\T_i}$ and $N_{\T_j}$ each contain a facet intersecting $F$. 
After relabeling, we may assume that $F$ intersects a facet in each of $N_{\T_1}$ and $N_{\T_2}$ and that 
$|N_{\T_i}|=|M_{\T_i}|$ for all $i=3,\ldots,d$.

Define $$W=N_{\T_1}\sqcup N_{\T_2}\sqcup M_{\T_3}\sqcup \cdots\sqcup M_{\T_d}.$$ Since all facets involved lie in pairwise disjoint components of $\T\setminus\{u\}$, the set $W$ is a $\nu$-matching of $\mathcal{N}(\T)$. Moreover, $F$ does not intersect any facet of $M_{\T_3}, \ldots, M_{\T_d}$, and hence $|\{H\in W : F\cap H\neq\emptyset\}|=2$ as required.
\end{proof}

\section{Componentwise linearity of highest non-vanishing squarefree power for trees} \label{sec.CL}

In this section, we establish the necessary and  sufficient condition for the componentwise linearity of highest non-vanishing squarefree power of the closed neighborhood ideal of a tree.

\vspace{2mm}

We introduce the following two conditions on a tree $\T$:

\begin{enumerate}[label=\bf{(C\arabic*)}, ref=C\arabic*]
\item \label{C1}
$\T$ does not contain a path 
$r_1, r, s, t, t_1$
with $\deg_{\T}(s) \ge 3$ such that the two distinct $\nu$-matchings
\[
\{\N_{\T}[r]\}\cup Y 
\quad \text{and} \quad 
\{\N_{\T}[t]\}\cup Y
\]
exist, where $Y$ is a $(\nu-1)$-matching of $\cN(\T)$.

\item \label{C2}
$\T$ does not contain a path $p_1=r,p_2,\ldots,p_{3n-1}=s$, for $n\ge 1$,
with $\deg_{\T}(r)\ge 3$ and $\deg_{\T}(s)\ge 3$, such that the sets
\[
\{\N_{\T}[p_1], \N_{\T}[p_4], \ldots, \N_{\T}[p_{3n-2}]\}\cup Y
\]
and
\[
\{\N_{\T}[p_2], \N_{\T}[p_5], \ldots, \N_{\T}[p_{3n-1}]\}\cup Y
\]
each form a $\nu$-matching of $\cN(\T)$,
where $Y$ is a $(\nu-n)$-matching of $\cN(\T)$.
\end{enumerate}

\vspace{2mm}

The following theorems show that Conditions~\eqref{C1} and \eqref{C2} are necessary for the $\nu$-th squarefree power of the closed neighborhood ideal of a tree to be componentwise linear.

\begin{theorem} \label{thm.1}
Let $\T$ be a tree that does not satisfy Condition~\eqref{C1}, and let $J=NI(\T)$. Then $J^{[\nu]}$ is not componentwise linear.
\end{theorem}

\begin{proof}
 Assume that $\T$ contains a path with edges $\{r_1,r\},\{r,s\},\{s,t\},\{t,t_1\}$ such that $\deg_{\T}(s) \ge 3$ and the monomials 
 \[
 {\bf x}_{\N_{\T}[r]}Y \quad \text{and } \quad {\bf x}_{\N_{\T}[t]}Y
 \]
 belong to $\mathcal{G}(J^{[\nu]})$, where $Y\in\mathcal{G}(J^{[\nu-1]})$. Thus $\T$ does not satisfy Condition~\eqref{C1}. Write $\N_{\T}[r]=\{r,s,r_1,r_2,\ldots,r_{d}\}$ and $\N_{\T}[t]=\{t,s,t_1,t_2,\ldots,t_{d'}\}$, where $d,d' \ge 1$. Set
\[
u={\bf x}_A ({\bf x}_{\N_{\T}[r]}Y )\quad \text{and} \quad v={\bf x}_B ({\bf x}_{\N_{\T}[t]}Y),
\]
where $A=\{t_{2},\dots,t_{d'}\}$ and $B=\{r_{2},\dots,r_{d}\}$, so that $\deg(u)=\deg(v)=a$. Let $I:=J^{[\nu]}$. Then $u,v \in I_{\langle a \rangle}$. By Theorem~\ref{Rem.LP}, it suffices to show that $u$ and $v$ are disconnected in the graph $G^{(u,v)}_{I_{\langle a \rangle}}$. Let $U=\{F,F_2 \ldots, F_\nu\}$ and $U'=\{G,F_2 \ldots, F_\nu\}$ be the $\nu$-matching corresponding to $u$ and $v$ respectively where $F=\N_{\T}[r]$, $G=\N_{\T}[t]$ and $\supp(Y)= \bigcup_{i=2}^{\nu}F_i$. By the construction of $u$ and $v$ , we have 
\[
\supp(v)\setminus \supp(u)=\{t,t_1\} \text{ and }  \supp(u)\setminus \supp(v)=\{r,r_1\},
\]
and therefore
\begin{equation}\label{eq=lcm}
\supp(\lcm(u,v))= \supp(u)\cup\{t,t_1\}=\supp(v)\cup\{r,r_1\}
= V_U \cup \N_{\T}[t] = V_{U'} \cup \N_{\T}[r]
\end{equation}

Let  $w\in G^{(u,v)}_{I_{\langle a \rangle}}$ and $W=\{H_1, \ldots, H_{\nu}\}$ be the $\nu$-matching corresponding to $w$. Removing the edge $\{s,t\}$ from $\T$ yields two connected components $\T_1$ and $\T_2$, where $\T_1$ contains $s$ and $\T_2$ contains $t$. Note that $\N_{\T}[s]$ and $\N_{\T}[t]$ are the only facets of $\mathcal{N}(\T)$ not entirely contained in either $V(\T_1)$ or $V(\T_2)$. Hence, \begin{equation}\label{eq}
\text{ if }\quad  \N_{\T}[s],\N_{\T}[t]\notin W, \quad\text{ then }\quad |U_{\T_1}|=|{W}_{\T_1}| \quad \text{and} \quad|U_{\T_2}|=|{W}_{\T_2}|.
\end{equation}
\noindent
\textit{Claim}. We have $\N_{\T}[s] \notin W$. 

\noindent {\it  Proof of claim: }  
    Suppose that $\N_{\T}[s] \in W$. Then $k=|U_{\T_1}|=|W_{\T_1}\cup\{\N_{\T}[s]\}|$, that is $U_{\T_1}$ and $W'=W_{\T_1}\cup\{\N_{\T}[s]\}$ gives a $k$-matching of $\T$. 
    
    We have $\N_{\T}[r]\in U_{\T_1}$. Following the Construction~\ref{cons:bipartite}, consider the bipartite graph $B_{(U_{\T_1},W_{\T_1})}$ on $2k'$ vertices where $2k'=2k-2t$ with $t=|U_{\T_1}\cap W_{\T_1}|$. Then $\{\N_{\T}[r],\N_{\T}[s]\}$ forms an edge in $B_{(U_{\T_1},W_{\T_1})}$ because $s \in \N_{\T}[r]\cap \N_{\T}[s]$. Due to our assumption $\deg_\T(s)\geq 3$, there exists $z \in \N_{\T}[s] \setminus \{r,s,t\}\subset V_W$. Since $ V_{W'} \subset V_{U_{\T_1}} \cup \{t\}$ due to (\ref{eq=lcm}), it follows that $z \in  V_{U_{\T_1}} $ and there exists a facet $F \in U_{\T_1}$ with $F\neq  \N_{\T_1}[r]$ such that $z \in F$. Hence, the degree of $\N_{\T_1}[s]$ in $B_{(U_{\T_1},W_{\T_1})}$ is at least $2$. Moreover, since $ V_{W_{\T_1}}\setminus\N_{\T_1}[s] \subset V_{U_{\T_1}}$, it follows from Remark \ref{rem=cons} that the degree of each vertex in partition set $W_{\T_1}$ of $B_{(U_{\T_1},W_{\T_1})}$  is also at least $2$. This shows that $B_{(U_{\T_1},W_{\T_1})}$ has more than $2k'$ edges, and it is not tree, a contradiction to Lemma \ref{lem.tree}. This proves our claim. 

Since $\deg(\lcm(u,v))=a+2$, it follows that  $\{u,v\}$ is not an edge in $G^{(u,v)}_{I_{\langle a \rangle}}$. In fact, given any $w \in G^{(u,v)}_{I_{\langle a \rangle}}$, if $\N_{\T}[r] \subset \supp(w)$, then $w$ is not adjacent to $v$. Then to prove that $u$ and $v$ are disconnected in $G^{(u,v)}_{I_{\langle a \rangle}}$, it is enough to prove that if a monomial $w \in G^{(u,v)}_{I_{\langle a \rangle}}$ is connected to $u$, then  $\N_{\T}[r] \in W$. We prove this by applying induction on $d(u,w)$, that is, the length of the shortest path in  $G^{(u,v)}_{I_{\langle a \rangle}}$ connecting $u$ and $v$. Note that the assertion trivially holds for $d(u,w)=0$ because in this case $u=w$. Now assume that assertion is true for all $d(u,w)\leq k-1$ where $k\geq 1$. Now, let $w \in G^{(u,v)}_{I_{\langle a \rangle}}$ be such that $d(u,w)= k$.

Since $w \in G^{(u,v)}_{I_{\langle a \rangle}}$, we have that  $\deg(\lcm(u,w))\leq a+2$. First consider the case when $\deg(\lcm(u,w))= a+1$. Then $u$ and $w$ differ in exactly one variable, that is, there exists some $x \in \supp(u)$ such that $w=(u/x)y$ where $y \in \{t,t_1\}$, due to (\ref{eq=lcm}). This shows that $\N_{\T}[t] \not\subset \supp(w)$ because $t$ or $t_1$ is missing from $\supp(w)$.  In particular, $\N_{\T}[t] \notin W$, and from (\ref{eq}) it follows that $|U_{\T_1}|=|{W}_{\T_1}|$. On the other hand, we have $V_{W_{\T_1}}\subset V_{U_{\T_1}}$ due to (\ref{eq=lcm}). Then by Proposition~\ref{Prop2.m=n}, we conclude that $U_{\T_1}={W}_{\T_1}$, and hence $\N_{\T_1}[r] \in W$ and $\N_{\T_1}[r] \subset \supp(w)$, as required. This also shows that $x$ is a vertex of $\T_2$.

Now, let $\deg(\lcm(u,w))= a+2$. Then $u$ and $w$ differ in exactly two variables, that is, there exists some $x,z \in \supp(u)$ such that $w=(u/xz)tt_1$ , due to (\ref{eq=lcm}). Since $w$ is connected to $u$, there exists a monomial $w' \in G^{(u,v)}_{I_{\langle a \rangle}}$ such that $w$ is adjacent to $w'$ and $d(u,w')=k-1$, and by the induction hypothesis $\N_{\T_1}[r]$ is contained in the matching corresponding to $w'$, and $w'=(u/x')y$, where $y\in \{t,t_1\}$ and $x'$ is in $\T_2$.

If $\N_{\T}[t]\notin W$, then again, as argued above, we obtain the desired conclusion. To complete the proof, it only remains to rule out the case $\N_{\T}[t]\in W$. Assume that $\N_{\T}[t]\in W$. If both $x$ and $z$ belong to $\T_1$, then $\deg(\lcm(v,w'))=a+2$ and $w$ cannot be adjacent to $w'$. So, at least one of the vertex, say $x$ belongs to $T_2$. Since $\N_{\T}[t]\in U'$, we obtain $|U'_{\T_2}|=|W_{\T_2}|$. Since $W_{\T_2}\subseteq U'_{\T_2}$ due to (\ref{eq=lcm}), by Proposition~\ref{Prop2.m=n}, we conclude that $U'_{\T_2}={W}_{\T_2}$. But $x\in V_{U'_{\T_2}}$, we obtain a contradiction to $x\notin \supp(w)$. 
\end{proof}

\begin{theorem} \label{thm.2}
    Let $\T$ be a tree that does not satisfy Condition \eqref{C2}, and let $J=NI(\T)$.
    Then $J^{[\nu]}$ is not componentwise linear.
\end{theorem}

\begin{proof}
 Let $r$ and $s$ be two vertices of $\T$ with $\deg(r)\geq \deg(s) \geq 3$. Assume that $P:p_1, \ldots, p_{3n-1}$ is a path on $3n-1$ vertices connecting $r$ and $s$, that is, $p_1=r$ and $p_{3n-1}=s$. Furthermore, assume that there exist monomials $u'={\bf x}_{\N_{\T}[p_1]}{\bf x}_{\N_{\T}[p_4]}\cdots {\bf x}_{\N_{\T}[p_{3n-2}]}{Y}$ and  $v'={\bf x}_{\N_{\T}[p_2]}{\bf x}_{\N_{\T}[p_5]}\cdots {\bf x}_{\N_{\T}[p_{3n-1}]}{Y}$ belong to $\G(J^{[\nu]})$, where $Y \in J^{[\nu -n]}$. Thus $\T$ does not satisfy Condition~\eqref{C2}.
 
 Let $P$ be a shortest path that satisfies the above condition. First, we claim that if $n\geq 2$ then one has $\deg_\T(p_i)=2$ for all $i \equiv 1,2\ \pmod{3}$ with $ i\in \{2,\dots,3n-2\}\}$. Suppose, to the contrary, that $\deg_\T(p_{i})\geq 3$ for some $i \in \{2,\ldots, 3n-2\}$ with $i \equiv 2 \pmod 3$. In this case, we have 
 \[u_1={\bf x}_{\N_{\T}[p_1]}{\bf x}_{\N_{\T}[p_4]}\cdots {\bf x}_{\N_{\T}[p_{i-1}]}{Z} \text{ and }
 v_1={\bf x}_{\N_{\T}[p_2]}{\bf x}_{\N_{\T}[p_5]}\cdots {\bf x}_{\N_{\T}[p_{i}]}{Z}\]
 belongs to $\G(J^{[\nu]})$, where $Z\in J^{[\nu - m]}$. This contradicts our assumption that $P$ is a shortest path that satisfies the hypothesis. Similarly, the case  $\deg_\T(p_{i})\geq 3$ for some $i \in \{2,\ldots, 3n-2\}$ with $i \equiv 1 \pmod 3$ is not possible. Hence, the claim follows.
    
Let $\N_{\T}[r]=\{r,r_1,\ldots,r_{d}\}$ and $\N_{\T}[s]=\{s,s_1,\ldots,s_{d'}\}$. If $n=1$, set $r_1=s$ and $s_1=r$; if $n\geq 2$, choose $r_1,s_1 \in V_P$, that is, $r_1=p_1$ and $s_1=p_{3n-2}$. Set
\[
u={\bf x}_A ({\bf x}_{\N_{\T}[r]}Y )\quad \text{and} \quad v={\bf x}_B ({\bf x}_{\N_{\T}[t]}Y),
\]
where $A=\{s_{4},\dots,s_{d'}\}$ and $B=\{r_{4},\dots,r_{d}\}$, so that $\deg(u)=\deg(v)=a$. Let $I:=J^{[\nu]}$. Then $u,v \in I_{\langle a \rangle}$. By Theorem~\ref{Rem.LP}, it suffices to show that $u$ and $v$ are disconnected in the graph $G^{(u,v)}_{I_{\langle a \rangle}}$. 

Let $U$ and $U'$ be the $\nu$-matching corresponding to $u$ and $v$. By the construction of $u$ and $v$, we have $
\supp(v)\setminus \supp(u)=\{s_2,s_3\} \text{ and }  \supp(u)\setminus \supp(v)=\{r_2,r_3\}, $
and therefore
\begin{equation}\label{eq=lcm2}
\supp(\lcm(u,v))= \supp(u)\cup\{s_2,s_3\}=\supp(v)\cup\{r_2,r_3\}
= V_U \cup \N_{\T}[s] = V_{U'} \cup \N_{\T}[r].
\end{equation}

Let $\T_1, \T_2, \ldots, \T_{d'}$ be the conncted components of $\T\setminus \{s\}$ such that $s_i$ lies in $T_i$ for each $i=1, \ldots, d'$. Let $w\in G^{(u,v)}_{I_{\langle a \rangle}}$ and $W=\{H_1, \ldots, H_{\nu}\}$ be the $\nu$-matching corresponding to $w$. Since $U$ and $W$ are maximal matching of $\T$, we note that 
    \begin{equation}\label{eq=size}
\text{ if} \quad \N_{\T}[s], \N_{\T}[s_i] \notin W, \quad \text{then} \quad   |{U}_{T_1 \cup \{s\}}| = |{W}_{T_1 \cup \{s\}}| 
     \end{equation}   
\noindent{\it Claim:} $\N_{\T}[s_i] \notin W$ for all $i=2,4,5, \ldots, d'$.

 \noindent Proof of claim:     Assume that $\N_{\T}[s_i] \in W$ for $i=2$. The same argument applies for any other choice of $i$. Let $F=\N_{\T}[s_1]$ and $G=\N_{\T}[s_2]$.
  
  Let $M=U_{\T_1} \cup U_{\T_2}\cup \{F\}\subseteq U$ and $M'=W_{\T_1} \cup W_{\T_2}\cup \{G\}\subseteq W$. Then $|M|=|M'|$ because $U$ and $W$ are maximal matching of $\T$.  Following the Construction~\ref{cons:bipartite}, consider the bipartite graph $B_{(M,M')}$ on $2k'$ vertices where $2k'=2|M|-2k$ with $k=|M\cap M'|$. Then $\{F,G\}$ forms an edge in $ B_{(M,M')}$, since $s \in F \cap G$. Since $G \in M'$, there exists $z \in G \setminus \{s,s_2\}\subset V_N$. Since $ V_{M'} \setminus A\subset V_{M} \cup \{s_2\}$, it follows that $z \in V_M$ and there exists a facet $F' \in M$ with $F'\neq F$ such that $z \in F'$. Hence, the degree of $G$ in $B_{(M,M')}$ is at least $2$. Moreover, since $V_{M'} \setminus \{G\} \subseteq V_{M}$, it follows from Remark \ref{rem=cons} that in $B_{(M,M')}$ the degree of each $G' \in M'$ is also at least $2$. This shows that $B_{(M,M')}$ has more than $2k'$ edges, and it is not tree, a contradiction to Lemma \ref{lem.tree}. This proves our claim.

We proceed as in the proof of Theorem~\ref{thm.1}. Since $\deg(\lcm(u,v))=a+2$, it follows that  $\{u,v\}$ is not an edge in $G^{(u,v)}_{I_{\langle a \rangle}}$. In fact, given any $w \in G^{(u,v)}_{I_{\langle a \rangle}}$, if $\N_{\T}[r] \subset \supp(w)$, then $w$ is not adjacent to $v$. Then to prove that $u$ and $v$ are disconnected in $G^{(u,v)}_{I_{\langle a \rangle}}$, it is enough to prove that if a monomial $w \in G^{(u,v)}_{I_{\langle a \rangle}}$ is connected to $u$, then  $\N_{\T}[r]\subset W$, the matching corresponding to $w$. We prove this by applying induction on $d(u,w)$. Note that the assertion trivially holds for $d(u,w)=0$ because in this case $u=w$. Now assume that assertion is true for all $d(u,w)\leq k-1$ where $k\geq 1$. Now, let $w \in G^{(u,v)}_{I_{\langle a \rangle}}$ be such that $d(u,w)= k$.

  Since $w \in G^{(u,v)}_{I_{\langle a \rangle}}$, we have that $\deg(\lcm(u,w))\leq a+2$. First, consider the case when $\deg(\lcm(u,w))= a+1$. Then, there exists some $x \in \supp(u)$ such that $w=(u/x)y$ where $y \in \{s_2,s_3\}$, due to (\ref{eq=lcm2}). This shows that $\N_{\T}[s] \not\subset \supp(w)$ because $s_2$ or $s_3$ is missing from $\supp(w)$.  In particular, $\N_{\T}[s] \notin W$, and from (\ref{eq=size}) it follows that $|U_{\T_1 \cup\{s\}}|=|{W}_{\T_1 \cup\{s\}}|$. On the other hand, we have $V_{W_{\T_1}}\subset V_{U_{\T_1}}$ due to (\ref{eq=lcm2}). Then by Proposition~\ref{Prop2.m=n}, we conclude that $U_{\T_1 \cup\{s\}}={W}_{\T_1 \cup\{s\}}$, and hence $\N_{\T_1}[r] \in W$, as required. This also shows that $x$ is a vertex of $\T_2$.

Now, let $\deg(\lcm(u,w))= a+2$. Then $u$ and $w$ differ in exactly two variables, that is, there exists some $x,z \in \supp(u)$ such that $w=(u/xz)s_2s_3$, due to (\ref{eq=lcm2}). Since $w$ is connected to $u$, there exists a monomial $w' \in G^{(u,v)}_{I_{\langle a \rangle}}$ such that $w$ is adjacent to $w'$ and $d(u,w')=k-1$, and by the induction hypothesis $\N_{\T_1}[r]$ is contained in the matching corresponding to $w'$, and $w'=(u/x')y$, where $y\in \{s_2,s_3\}$ and $x'$ is in $\T_2$.

If $\N_{\T}[s]\notin W$, then again, as argued above, we obtain the desired conclusion. To complete the proof, it only remains to rule out the case $\N_{\T}[s]\in W$. Assume that $\N_{\T}[s]\in W$. Since $\N_{\T}[s]\in U'$, we obtain $|U'_{\T_2}|=|W_{\T_2}|$. We also have $W_{\T_2}\subseteq U'_{\T_2}$ due to (\ref{eq=lcm2}), by Proposition~\ref{Prop2.m=n}, we conclude that $U'_{\T_2}={W}_{\T_2}$. This shows that both $x$ and $z$ belong to $\T_1$. This gives $\deg(\lcm(v,w'))=a+2$ and $w$ cannot be adjacent to $w'$,  a contradiction.
\end{proof}

Next, we show that (\ref{C1}) and (\ref{C2}) are necessary for componentwise linearity of  highest non-vanishing squarefree power of $NI(\T)$. In fact, we show that if $\T$ satisfies (\ref{C1}) and (\ref{C2}), then $NI(\T)^{[\nu]}$ admits linear quotient. To do this, we now define two total orders on the $\nu$-matchings of $\cN(\T)$, namely $>_{\lex}$ and $>_\ell$.

\begin{notation}\label{Notation.facet}
Let $\T$ be a tree and let $x$ be a pendent vertex of $\T$. We view $\T$ as a rooted tree with root $x$. For a vertex $u\in V(\T)$, we define $\level(u)$ to be the distance from $x$ to $u$. We assign the following notation to the facets of $\mathcal{N}(\T)$. For a vertex $u\in V(\T)$,  we write $F_{i,j,k}=N_{\T}[u]$ if $|N_{\T}[u]|=i$, $\level(u)=j$, and $u$ is the $k$-th vertex in a fixed ordering of the vertices at level $j$ (see Remark~\ref{rem:order}).

We define a total order on $\mathcal{F}(\mathcal{N}(\T))$ by setting $F_{i,j,k} > F_{i',j',k'}$ if one of the following holds:
\begin{enumerate}[label=(\roman*)]
\item $i<i'$;
\item $i=i'$ and $j<j'$;
\item $i=i'$, $j=j'$, and $k<k'$.
\end{enumerate}

Let $ M=\{F_{i_1,j_1,k_1},\ldots,F_{i_\nu,j_\nu,k_\nu}\}$ and $N=\{F_{i'_1,j'_1,k'_1},\ldots,F_{i'_\nu,j'_\nu,k'_\nu}\}$ be two $\nu$-matchings of $\mathcal{N}(\T)$, written in decreasing order with respect to the above facet order. Let
$s=\min\{\ell : F_{i_\ell,j_\ell,k_\ell}\neq F_{i'_\ell,j'_\ell,k'_\ell}\}.$ We define a lexicographic order on $\nu$-matchings by setting $M>_{\mathrm{lex}}N$ if $F_{i_s,j_s,k_s} > F_{i'_s,j'_s,k'_s}.$

By Corollary~\ref{cor.unique}, each generator of $NI(\T)^{[\nu]}$ corresponds uniquely to a $\nu$-matching of $\mathcal{N}(\T)$, and this induces a total order $>_{\mathrm{lex}}$ on the minimal generating set $\mathcal{G}(NI(\T)^{[\nu]})$.
\end{notation}
\begin{remark}\label{rem:order}
The index $k$ is used only to impose a total order on facets corresponding to vertices with the same value of $\level(u)$ and with neighborhoods of the same size; any fixed choice of $k$ suffices.
\end{remark}

\begin{example}\label{ex.facetorder}
Let $\T$ be the tree on $[14]$ shown in Figure~\ref{fig.G}, rooted at the pendent vertex $1$. Using Notation~\ref{Notation.facet}, we label the facets of the closed neighborhood complex $\mathcal{N}(\T)$ as follows:
\[
\begin{array}{lll}
F_{2,0,1}=N_G[1]=\{1,2\},    \quad  &
F_{2,6,1}=N_G[14]=\{13,14\}, \quad &
F_{2,10,1}=N_G[11]=\{10,11\}, \\[2pt]
F_{3,2,1}=N_G[3]=\{2,3,4\},   &
F_{3,4,1}=N_G[12]=\{4,12,13\},&
F_{3,4,2}=N_G[5]=\{4,5,6\},   \\[2pt]
F_{3,5,2}=N_G[6]=\{5,6,7\},   &
F_{3,6,2}=N_G[7]=\{6,7,8\},   &
F_{3,7,1}=N_G[8]=\{7,8,9\},   \\[2pt]
F_{3,8,1}=N_G[9]=\{8,9,10\},  &
F_{4,3,1}=N_G[4]=\{3,4,5,12\}.&
\end{array}
\]
According to the total order on facets defined in Notation~\ref{Notation.facet}, we obtain
\[
F_{2,0,1} > F_{2,6,1} > F_{2,10,1} > F_{3,2,1} > F_{3,4,1} > F_{3,4,2}
> F_{3,5,2} > F_{3,6,2} > F_{3,7,1} > F_{3,8,1} > F_{4,3,1}.
\]

The $5$-matchings of $\mathcal{N}(\T)$ are
\[
M_1=\{F_{2,0,1},F_{2,6,1},F_{2,10,1},F_{3,4,2},F_{3,7,1}\},
\]
\[
M_2=\{F_{2,0,1},F_{2,6,1},F_{2,10,1},F_{3,6,2},F_{4,3,1}\},
\]
\[
M_3=\{F_{2,0,1},F_{2,6,1},F_{2,10,1},F_{3,7,1},F_{4,3,1}\}.
\]
With respect to the lexicographic order on $\nu$-matchings introduced in Notation~\ref{Notation.facet}, we have
\[
M_1 >_{\mathrm{lex}} M_2 >_{\mathrm{lex}} M_3.
\]
\end{example}

\begin{figure}[ht]
    \centering
    
\tikzset{every picture/.style={line width=0.75pt}} 

\begin{tikzpicture}[x=0.5pt,y=0.5pt,yscale=-1,xscale=1]
\draw    (172.03,1719.03) -- (203.03,1700.27) ;
\draw    (203.03,1700.27) -- (234.03,1681.5) ;
\draw    (234.03,1681.5) -- (265.03,1662.73) ;
\draw    (265.03,1662.73) -- (296.03,1643.97) ;
\draw    (296.03,1643.97) -- (327.03,1625.2) ;
\draw    (327.03,1625.2) -- (358.03,1606.43) ;
\draw    (358.03,1606.43) -- (389.03,1587.67) ;
\draw    (389.03,1587.67) -- (420.03,1568.9) ;
\draw    (420.03,1568.9) -- (451.03,1550.13) ;
\draw    (141.03,1737.8) -- (172.03,1719.03) ;
\draw    (358.03,1606.43) -- (383.03,1634.27) ;
\draw    (383.03,1634.27) -- (408.03,1662.1) ;
\draw    (408.03,1662.1) --(433.03,1689.93) ;

\draw (431.03,1535.13) node [anchor=north west][inner sep=0.75pt]  [font=\footnotesize]  {$1$};
\draw (405.03,1552.9) node [anchor=north west][inner sep=0.75pt]  [font=\footnotesize]  {$2$};
\draw (376.03,1568.67) node [anchor=north west][inner sep=0.75pt]  [font=\footnotesize]  {$3$};
\draw (344.03,1588.43) node [anchor=north west][inner sep=0.75pt]  [font=\footnotesize]  {$4$};
\draw (155.03,1697.03) node [anchor=north west][inner sep=0.75pt]  [font=\footnotesize]  {$10$};
\draw (191.03,1682.27) node [anchor=north west][inner sep=0.75pt]  [font=\footnotesize]  {$9$};
\draw (222.03,1663.5) node [anchor=north west][inner sep=0.75pt]  [font=\footnotesize]  {$8$};
\draw (252.03,1647.73) node [anchor=north west][inner sep=0.75pt]  [font=\footnotesize]  {$7$};
\draw (280.03,1627.97) node [anchor=north west][inner sep=0.75pt]  [font=\footnotesize]  {$6$};
\draw (309.03,1608.2) node [anchor=north west][inner sep=0.75pt]  [font=\footnotesize]  {$5$};
\draw (121.03,1717.8) node [anchor=north west][inner sep=0.75pt]  [font=\footnotesize]  {$11$};
\draw (362,1635) node [anchor=north west][inner sep=0.75pt]  [font=\footnotesize]  {$12$};
\draw (385,1661) node [anchor=north west][inner sep=0.75pt]  [font=\footnotesize]  {$13$};
\draw (415,1695) node [anchor=north west][inner sep=0.75pt]  [font=\footnotesize]  {$14$};

\filldraw[black] (172.03,1719.03)  circle (1.5pt) ;
\filldraw[black] (203.03,1700.27) circle (1.5pt) ;
\filldraw[black] (234.03,1681.5) circle (1.5pt) ;
\filldraw[black] (265.03,1662.73)  circle (1.5pt) ;
\filldraw[black] (296.03,1643.97) circle (1.5pt) ;
\filldraw[black] (327.03,1625.2) circle (1.5pt) ;\filldraw[black] (358.03,1606.43)  circle (1.5pt) ;
\filldraw[black] (389.03,1587.67) circle (1.5pt) ;
\filldraw[black] (420.03,1568.9)  circle (1.5pt) ;
\filldraw[black] (451.03,1550.13) circle (1.5pt) ;
\filldraw[black] (141.03,1737.8) circle (1.5pt) ;\filldraw[black] (383.03,1634.27)  circle (1.5pt) ;
\filldraw[black] (408.03,1662.1) circle (1.5pt) ;
\filldraw[black] (433.03,1689.93) circle (1.5pt) ;

\end{tikzpicture}    \caption{The tree $\T$.} \label{fig.G}
\end{figure}

Let $U$ be the unique $\nu$-matching which is maximal with respect to $>_{\mathrm{lex}}$ (Notation~\ref{Notation.facet}). For any $\nu$-matching $M$ of $\cN(\T)$, we define the level of $M$ with respect to $U$, denoted by $\level_{U}(M)$, by 
$$\level_{U}(M) :=\nu - |U\cap M|.$$
Furthermore, for a $\nu$-matching $M$, we denote by 
\[
\beta(M):=\bigl|\{F\in M :\; |F|=2\}\bigr|
\]
the number of facets of $M$ of cardinality two.

\begin{notation}\label{notation.total}
We define a total order $>_\ell$ on the set of $\nu$-matchings of $\mathcal{N}(\T)$ as follows. For two $\nu$-matchings $M$ and $N$, we set $M>_\ell N$ if one of the following holds:
\begin{enumerate}[label=(\roman*)]
\item $\beta(M)>\beta(N)$;
\item $\beta(M)=\beta(N)$ and $\level_{U}(M)<\level_{U}(N)$;
\item $\beta(M)=\beta(N)$, $\level_{U}(M)=\level_{U}(N)$, and $M>_{\mathrm{lex}}N$ (see Notation~\ref{Notation.facet}).
\end{enumerate}

Via the bijection between $\nu$-matchings of $\mathcal{N}(\T)$ and the minimal generators of $NI(\T)^{[\nu]}$ (Corollary~\ref{cor.unique}), this order induces a total order on $\mathcal{G}(NI(\T)^{[\nu]})$, which we also denote by $>_\ell$.
\end{notation}

\begin{example}
Let $\T$ be the tree from Example~\ref{ex.facetorder}, and let $M_1,M_2,M_3$ be the $\nu$-matchings of $\mathcal{N}(\T)$ described there.
Then $M_1$ is the maximal matching with respect to $>_{\lex}$ and we have $\level_{U}(M_3)=1 \quad \text{and} \quad \level_{U}(M_2)=2.$

With respect to the total order $>_\ell$ defined in Notation~\ref{notation.total}, we obtain
\[
M_1 >_\ell M_3 >_\ell M_2.
\]
On the other hand, with respect to the lexicographic order $>_{\mathrm{lex}}$ (see Example~\ref{ex.facetorder}), we have
\[
M_1 >_{\mathrm{lex}} M_2 >_{\mathrm{lex}} M_3.
\]
This example illustrates that the order $>_\ell$ refines the lexicographic order by incorporating both the number of size-two facets and the number of common facets with $U$.
\end{example}


In the subsequent results, we describe the structural properties of $\nu$-matchings of $\cN(\T)$ that satisfies $\eqref{C1}$ and $\eqref{C2}$.

\begin{remark}\label{rem=deg2}Let $\T$ be a tree satisfying (\ref{C1}). Let $M$ and $N$ be two $\nu$-matchings of $\cN(\T)$. Suppose that 
$\N_{\T}[r]\in M$ and $\N_{\T}[t]\in N$ satisfy 
$\N_{\T}[r]\cap \N_{\T}[t]=\{s\}$ and $\deg_\T(r), \deg_\T(t) \geq 2$. 
\begin{enumerate}
    \item  

Let $\T_1,\ldots,\T_d$ be the connected components of 
$\T\setminus\{r\}$, and assume that $t\in \T_1$. 
If $|N_{\T_1}|=|M_{\T_1}|+1$, then define
\[
Y = \bigl(M\setminus M_{\T_1}\bigr) 
    \,\cup\, 
    \bigl(N_{\T_1}\setminus\{\N_{\T}[t]\}\bigr).
\]
Then $Y$ is a $(\nu-1)$-matching of $\cN(\T)$, and both 
$\{\N_{\T}[r]\}\cup Y$ and 
$\{\N_{\T}[t]\}\cup Y$ 
are $\nu$-matchings of $\cN(\T)$. 
Hence, by (\ref{C1}), it follows that $\deg_{\T}(s)=2$.
\item If $\N_{\T}[r]$ does not intersect any other facet of $N$. Then $Y = (N\setminus \N_{\T}[t])$ is a $(\nu-1)$-matching of $\cN(\T)$, and both 
$\{\N_{\T}[r]\}\cup Y$ and 
$\{\N_{\T}[t]\}\cup Y=N$ 
are $\nu$-matchings of $\cN(\T)$. 
Hence, by (\ref{C1}), it follows that $\deg_{\T}(s)=2$.
\end{enumerate}
\end{remark}


\begin{proposition}\label{prop:4face}
Let $\T$ be a tree satisfying \eqref{C1} and \eqref{C2}. 
Let $U$ be the unique $\nu$-matching that is maximal with respect to $>_{\mathrm{lex}}$ (Notation~\ref{Notation.facet}). Then every facet $F\in U$ with $|F|\ge 4$ belongs to every $\nu$-matching $N$ of $\mathcal{N}(\T)$.
\end{proposition}

\begin{proof}
Let $F=\N_{\T}[i_1]\in U$, where $\deg_{\T}(i_1)=d\ge 3$. 
Let $\T_1,\ldots,\T_d$ be the connected components of $\T\setminus\{i_1\}$. Suppose, for a contradiction, that there exists a $\nu$-matching $N$ such that 
$\N_{\T}[i_1]\notin N$. We consider two cases: (1) $i_1\notin V_N$, \quad (2) $i_1\in V_N$.

\medskip
\noindent
\textit{Case I.} Suppose that $i_1\notin V_N$. 
By Lemma~\ref{lemma1.redu}, there exists $j\in[d]$, say $j=1$, such that 
$|N_{\T_1}|=|U_{\T_1}|+1$, and a facet $G\in N_{\T_1}$ with 
$G\cap \N_{\T}[i_1]\neq\emptyset$. Write $G=\N_{\T}[i_3]$ and assume 
$\N_{\T}[i_1]\cap G=\{i_2\}$, where $i_1,i_2,i_3$ form a path in $\T$. 
By Remark~\ref{rem=deg2}(1), we have $\deg_{\T}(i_2)=2$.
\begin{figure}[ht]
    \centering
\tikzset{every picture/.style={line width=0.75pt}} 
\begin{tikzpicture}[x=0.75pt,y=0.75pt,yscale=-1,xscale=1]

\draw    (181.03,1468.03) -- (230.43,1467.82) ;
\draw    (230.43,1467.82) -- (279.83,1467.6) ;
\draw    (279.83,1467.6) -- (329.23,1467.38) ;
\draw    (329.23,1467.38) -- (378.63,1467.17) ;

\draw    (329.23,1467.38) -- (350.63,1430) ;
\draw    (329.23,1467.38) -- (300.63,1430) ;
\draw    (192.03,1431.03) -- (230.43,1467.82) ;
\draw    (364.82,1428.4) -- (378.63,1467.17) ;
\draw    (413.82,1431.4) -- (378.63,1467.17) ;
\draw    (426.82,1466.92) -- (378.63,1467.17) ;
\draw    (470,1467.17) -- (426.82,1467.17)  ;
\draw    (470,1467.17) -- (518,1467.17)  ;
\draw    (470,1467.17) -- (440,1430) ;
\draw    (470,1467.17) -- (490,1430) ;

\draw (180,1460) node [anchor=north west][inner sep=0.75pt]  [rotate=68,xslant=0]  {$\cdots $};
\draw (231.23,1469.42) node [anchor=north west][inner sep=0.75pt]  [font=\footnotesize]  {$i_{1}$};
\draw (280.83,1467.6) node [anchor=north west][inner sep=0.75pt]  [font=\footnotesize]  {$i_{2}$};
\draw (378.63,1467.17) node [anchor=north west][inner sep=0.75pt]  [font=\footnotesize]  {$i_{4}$};
\draw (329.23,1467.38) node [anchor=north west][inner sep=0.75pt]  [font=\footnotesize]  {$i_{3}$};
\draw (426.82,1467.92) node [anchor=north west][inner sep=0.75pt]  [font=\footnotesize]  {$i_{5}$};

\draw (471,1467.92) node [anchor=north west][inner sep=0.75pt]  [font=\footnotesize]  {$i_{6}$};

\draw (530,1465) node [anchor=north west][inner sep=0.75pt]  [font=\footnotesize]  {$\ldots$};

\draw (155,1465) node [anchor=north west][inner sep=0.75pt]  [font=\footnotesize]  {$\ldots$};

\draw (378.63,1430) node [anchor=north west][inner sep=0.75pt]  [font=\footnotesize]  {$\cdots$};

\draw (460,1430) node [anchor=north west][inner sep=0.75pt]  [font=\footnotesize]  {$\cdots$};
\draw (320,1430) node [anchor=north west][inner sep=0.75pt]  [font=\footnotesize]  {$\cdots$};

\filldraw[black](440,1430) circle (1.5pt) ;
  \filldraw[black] (490,1430) circle (1.5pt) ;

\filldraw[black](181.03,1468.03)  circle (1.5pt) ;
\filldraw[red] (230.43,1467.82) circle (1.5pt) ;
\filldraw[black] (279.83,1467.6) circle (1.5pt) ;
\filldraw[blue] (329.23,1467.38)  circle (1.5pt);
\filldraw[red] (378.63,1467.17) circle (1.5pt) ;
\filldraw[black] (192.03,1431.03) circle (1.5pt) ;\filldraw[black] (364.82,1428.4)  circle (1.5pt) ;
\filldraw[black] (413.82,1431.4) circle (1.5pt) ;
\filldraw[black] (426.82,1466.92) circle (1.5pt) ;
\filldraw[blue] (470,1467.17) circle (1.5pt) ;

\filldraw[red] (518,1467.17) circle (1.5pt) ;
\filldraw[black] (350.63,1430) circle (1.5pt) ;\filldraw[black] (300.63,1430) circle (1.5pt) ;
\end{tikzpicture}
   \caption{The neighborhoods of the red vertices belong to $V_U$, 
while the neighborhoods of the blue vertices belong to $V_N$.}
    \label{fig:case1}
\end{figure}
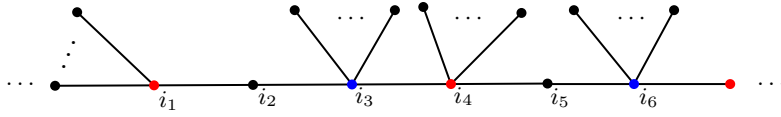

If $i_3\notin V_U$, then 
\[
U'=(U\setminus \N_{\T}[i_1])\cup \N_{\T}[i_2]\]
is a $\nu$-matching of $\T$. Since $\N_{\T}[i_2]>\N_{\T}[i_1]$ with respect to the order in Notation~\ref{Notation.facet}, we obtain $U'>_{\lex}U$ and this contradicts
the maximality of $U$. Hence $i_3\in V_U$, and there exists a neighbor
$i_4$ of $i_3$ with $\N_{\T}[i_4]\in U$ and 
$\N_{\T}[i_2]\cap \N_{\T}[i_4]=\{i_3\}$.
Moreover, $i_4$ is not a leaf of $\T$, since otherwise
$\N_{\T}[i_4]\subsetneq \N_{\T}[i_3]$, contradicting that
$\N_{\T}[i_3]$ is a facet.

Let $\Sc_1,\ldots,\Sc_t$ be the connected components of 
$\T\setminus\{i_4\}$, and assume that $i_3\in \Sc_1$. 
Since $|N_{\T_1}|=|U_{\T_1}|+1$, we have $|N_{\Sc_1}|=|U_{\Sc_1}|-1$. 
By Lemma~\ref{lemma1.redu}(ii), there exists $j\in[t]$, say $j=2$, such that 
$|N_{\Sc_2}|=|U_{\Sc_2}|+1$. Furthermore, by Lemma~\ref{lemma1.redu}(i), there exists a facet $\N_{\T}[i_6]\in N_{\Sc_2}$ such that $\N_{\T}[i_4]\cap \N_{\T}[i_6]=\{i_5\}$ (see Figure~\ref{fig:case1}). 
By Remark~\ref{rem=deg2}(1), it follows that $\deg_{\T}(i_5)=2$.

If $i_6 \notin V_U$, then the set
\[
U''=\bigl(U\setminus \{\N_{\T}[i_1],\N_{\T}[i_4]\}\bigr)
      \cup \{\N_{\T}[i_2],\N_{\T}[i_5]\}
\]
is a $\nu$-matching of $\mathcal{N}(\T)$. Since $i_4$ lies between $i_2$ and $i_5$ in $\T$ and $|\N_{\T}[i_2]|=|\N_{\T}[i_5]|=3$, it follows from 
Notation~\ref{Notation.facet} that either $\N_{\T}[i_2]>\N_{\T}[i_4]$ or 
$\N_{\T}[i_5]>\N_{\T}[i_4]$. 
Consequently, $U''>_{\lex}U$, contradicting the maximality of $U$. Hence $i_6\in V_U$. 

Repeating the same argument at $i_6$, 
we extend the path $i_1,i_2,i_3,i_4,i_5,i_6,\ldots $ strictly farther from $i_1$, where every second vertex has degree $2$
and the terminal vertex lies in $V_U$. Since $\T$ is finite,
this process must reach a leaf, where the construction cannot continue.
At that stage the alternative $i_{3k}\notin V_U$ must occur, yielding a contradiction to the 
maximality of $U$. Therefore, the initial assumption $\N_{\T}[i_1]\notin N$ is impossible.

\textit{Case II.} Suppose that $i_1\in V_N$. Then there exists a vertex 
$i_2$ adjacent to $i_1$ such that $\N_{\T}[i_2]\in N$. 
Assume $i_2\in V(\T_1)$. 
By Lemma~\ref{lemma1.redu}, one of the following occurs:

\begin{enumerate}
  \item[(i)] There exists $j$, say $j=2$, such that 
  $|N_{\T_2}|=|U_{\T_2}|+1$. 
  Since $i_2\in V(\T_1)$, it follows from the proof of 
  Lemma~\ref{lemma1.redu}(ii) that 
  $|N_{\T_1}|=|U_{\T_1}|-1$, and 
  $|N_{\T_k}|=|U_{\T_k}|$ for all $k\neq 1,2$.
  
  \item[(ii)] $|N_{\T_j}|=|U_{\T_j}|$ for all $j$.
\end{enumerate}

Suppose that {\rm (i)} holds. Then $M=(U\setminus \{\N_{\T}[i_1]\}) 
  \cup N_{\T_2}$ is a $\nu$-matching of $\mathcal{N}(\T)$, since the componentwise
cardinalities agree. Moreover, $\N_{\T}[i_1]\notin M$ and 
$i_1\notin V_M$. As $|\N_{\T}[i_1]|\ge 4$, this contradicts Case~I,
which shows that no such $\nu$-matching can exist. Hence {\rm (i)} does not occur. Therefore $|N_{\T_j}|=|U_{\T_j}|$ for all $j$. Set
\[
Y= N_{\T_1}\cup U_{\T_2}\cup \cdots \cup U_{\T_d}.
\]
Then $Y$ is a $(\nu-1)$-matching of $\mathcal{N}(\T)$, and both $Y\cup \N_{\T}[i_1]$ and $Y\cup \N_{\T}[i_2]$ are $\nu$-matchings of $\mathcal{N}(\T)$.
These two matchings are distinct, since $i_1\neq i_2$.
Because $\T$ satisfies {\rm (C2)} and $\deg_{\T}(i_1)\ge 3$,
it follows that $\deg_{\T}(i_2)=2$.  Since $Y\cup \N_{\T}[i_2]$ satisfies the same assumptions as $N$ in Case~II, we may set $N=Y\cup \N_{\T}[i_2]$ in the rest of the proof.

Since $\N_{\T}[i_2]>\N_{\T}[i_1]$ with respect to the order in 
Notation~\ref{Notation.facet}, the maximality of $U$ implies that 
$\N_{\T}[i_2]\cap V_{U\setminus\{\N_{\T}[i_1]\}}\neq\emptyset$; 
otherwise replacing $\N_{\T}[i_1]$ by $\N_{\T}[i_2]$ would yield a 
$\nu$-matching strictly larger than $U$, a contradiction. Therefore, there exists a facet $\N_\T[i_4]\in U\setminus \N_\T[i_1]$ such that $\N_\T[i_2] \cap \N_\T[i_4] = \{i_3\}$. 

Suppose that $i_4\notin V_N$. 
If $\deg_{\T}(i_4)>2$, then Case~I applied at $i_4$ yields 
$\N_{\T}[i_4]\in N$, a contradiction; hence $\deg_{\T}(i_4)\le 2$.
If $\N_{\T}[i_4]$ intersects only $\N_{\T}[i_2]$ among the facets of $N$, 
then replacing $\N_{\T}[i_2]$ by $\N_{\T}[i_4]$ and $\N_{\T}[i_1]$ produces a 
$(\nu+1)$-matching, a contradiction. 
Thus $\N_{\T}[i_4]$ must intersect another facet of $N$, and we obtain $\deg_{\T}(i_4)=2$. Let $i_5$ be the neighbor of $i_4$ distinct from $i_3$, and let 
$\N_{\T}[i_6]\in N$ be the facet intersecting $\N_{\T}[i_4]$, so that 
$i_4,i_5,i_6$ form a path in $\T$. Then $\N_{\T}[i_4]$ intersects only $\N_{\T}[i_6]$ among the facets of 
$W'=(U\setminus U_{\T_1})\cup N_{\T_1}$. 
By Condition~\eqref{C1}, we obtain $\deg_{\T}(i_5)=2$. Since $i_4\notin V_N$, replacing $\N_{\T}[i_6]$ in $N$ by 
$\N_{\T}[i_5]$ yields a $\nu$-matching $N'$ satisfying the same 
assumptions as $N$ in Case~II. Hence we may replace $N$ by $N'$ and 
continue; in particular, we may assume $i_4\in V_N$.

Now assume that $i_4\in V_N$. Then there exists a vertex $i_5$ such that 
$i_4\in \N_{\T}[i_5]\in N$. Note that $i_5$ cannot be a leaf of $\T$, 
since otherwise $\N_{\T}[i_5]\subsetneq \N_{\T}[i_4]$, contradicting the 
fact that $\N_{\T}[i_4]$ is a facet of $\mathcal{N}(\T)$. We distinguish according to the degree of $i_4$. If $\deg_{\T}(i_4)\ge 3$, then the same argument used at the beginning 
of Case~II (applied now to the pair $i_4,i_5$ in place of $i_1,i_2$) 
and Condition \eqref{C2} yields $\deg_{\T}(i_5)=2$. If $\deg_{\T}(i_4)=2$, then $i_1,\ldots,i_5$ form a path in $\T$, and 
the set 
\[
M = N \setminus \{\N_{\T}[i_2], \N_{\T}[i_5]\}
\]
is a $(\nu-2)$-matching such that both 
$M\cup\{\N_{\T}[i_2],\N_{\T}[i_5]\}=N$ and 
$M\cup\{\N_{\T}[i_1],\N_{\T}[i_4]\}$ are $\nu$-matchings. 
Since $\T$ satisfies  \eqref{C2} and $\deg_{\T}(i_1)\ge 3$, it follows again that 
$\deg_{\T}(i_5)=2$. Thus in all cases $\deg_{\T}(i_5)=2$. 
Because $i_4$ lies between $i_2$ and $i_5$ in $\T$, it follows from 
Notation~\ref{Notation.facet} that either 
$\N_{\T}[i_2]>\N_{\T}[i_4]$ or $\N_{\T}[i_5]>\N_{\T}[i_4]$. 
Consequently,
\[
(U\setminus \{\N_{\T}[i_1],\N_{\T}[i_4]\})
\cup \{\N_{\T}[i_2],\N_{\T}[i_5]\}
>_{\lex} U,
\]
contradicting the maximality of $U$.

\begin{figure}[ht]
\centering
\tikzset{every picture/.style={line width=0.7pt}} 
\begin{tikzpicture}[x=0.75pt,y=0.75pt,yscale=-1,xscale=1]
\draw    (220.03,1486.03) -- (269.43,1485.82) ;
\draw    (269.43,1485.82) -- (318.83,1485.6) ;
\draw    (318.83,1485.6) -- (368.23,1485.38) ;
\draw    (368.23,1485.38) -- (417.63,1485.17) ;
\draw    (245.03,1449.85) -- (269.43,1485.82) ;
\draw    (399.03,1450.85) -- (417.63,1485.17) ;
\draw    (437.03,1450.85) -- (417.63,1485.17) ;
\draw    (465.82,1484.92) -- (417.63,1485.17) ;
\draw    (525,1450.85) -- (514,1484.67) ;
\draw    (490,1450.85) -- (514,1484.67) ;
\draw    (514,1484.67) -- (465.82,1484.92) ;

\draw (368.23,1485.38) -- (385,1450.85);
\draw (368.23,1485.38) -- (350,1450.85);
\draw (514,1484.67) -- (555,1484.67);

\draw (270.23,1487.42) node [anchor=north west][inner sep=0.75pt]  [font=\footnotesize]  {$i_{1}$};
\draw (319.83,1485.6) node [anchor=north west][inner sep=0.75pt]  [font=\footnotesize]  {$i_{2}$};
\draw (417.63,1485.17) node [anchor=north west][inner sep=0.75pt]  [font=\footnotesize]  {$i_{4}$};
\draw (368.23,1485.38) node [anchor=north west][inner sep=0.75pt]  [font=\footnotesize]  {$i_{3}$};
\draw (465.82,1484.92) node [anchor=north west][inner sep=0.75pt]  [font=\footnotesize]  {$i_{5}$};
\draw (514,1484.67) node [anchor=north west][inner sep=0.75pt]  [font=\footnotesize]  {$i_{6}$};

\draw (555,1484.67) node [anchor=north west][inner sep=0.75pt]  [font=\footnotesize]  {$i_{7}$};

\draw (408,1450) node [anchor=north west][inner sep=0.75pt]  [font=\footnotesize]  {$\dotsc $};
\draw (500,1450) node [anchor=north west][inner sep=0.75pt]  [font=\footnotesize]  {$\dotsc $};
\draw (560,1482) node [anchor=north west][inner sep=0.75pt]  [font=\footnotesize]  {$\dotsc $};
\draw (193.23,1482) node [anchor=north west][inner sep=0.75pt]  [font=\footnotesize]  {$\dotsc $};

\draw (358,1450.85) node [anchor=north west][inner sep=0.75pt]  [font=\footnotesize]  {$\dotsc $};

\draw (220,1472.22) node [anchor=north west][inner sep=0.75pt]  [font=\footnotesize,rotate=-302.13]  {$\dotsc $};

\filldraw[black](220.03,1486.03)  circle (1.5pt) ;
\filldraw[red](269.43,1485.82)  circle (1.5pt) ;
\filldraw[blue](318.83,1485.6) circle (1.5pt) ;\filldraw[black](368.23,1485.38)  circle (1.5pt) ;\filldraw[red](417.63,1485.17) circle (1.5pt) ;
\filldraw[black](245.03,1449.85)  circle (1.5pt) ;
 \filldraw[black](399.03,1450.85) circle (1.5pt) ;
 \filldraw[black](437.03,1450.85) circle (1.5pt) ;
  \filldraw[blue](465.82,1484.92)circle (1.5pt) ;
\filldraw[black](525,1450.85)  circle (1.5pt) ;\filldraw[black] (490,1450.85) circle (1.5pt) ;
\filldraw[black] (514,1484.67)  circle (1.5pt) ;

\filldraw[black] (385,1450.85) circle (1.5pt) ;
  \filldraw[black] (350,1450.85)  circle (1.5pt) ;
  \filldraw[red](555,1484.67) circle (1.5pt) ;

\end{tikzpicture}

  \caption{The neighborhoods of the red vertices belong to $V_U$, 
while the neighborhoods of the blue vertices belong to $V_N$.}
    \label{fig:case2}
\end{figure}
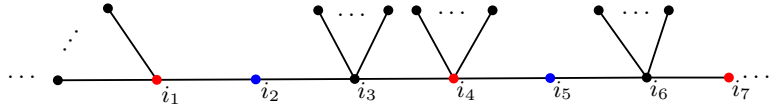

Thus $\N_{\T}[i_5]\cap V_{U\setminus\{\N_{\T}[i_4]\}}\neq\emptyset$, 
so there exists a facet $\N_{\T}[i_7]\in U$ intersecting $\N_{\T}[i_5]$. 
Repeating the same argument alternately with facets of $N$ and $U$, 
we extend the simple path $i_1,i_2,i_3,i_4,i_5,i_6,i_7,\ldots $ strictly farther away from $i_1$ in $\T$, while each newly created 
intermediate vertex has degree $2$. Since $\T$ is finite, this process must terminate, and at the terminal step the alternative $\N_{\T}[i_{3k+2}]\cap V_{U\setminus\{\N_{\T}[i_{3k+1}]\}}\neq\emptyset $ can no longer hold. Hence the complementary case must occur, 
which yields the contradiction to the maximality of $U$. 
Therefore $\N_{\T}[i_1]\in N$, contradicting our assumption.
\end{proof}


\begin{proposition}\label{prop:level}
Let $\T$ be a tree satisfying  \eqref{C1} and  \eqref{C2}. 
Let $U$ be the unique $\nu$-matching that is maximal with respect to 
$>_{\mathrm{lex}}$ (Notation~\ref{Notation.facet}), 
and let $N$ be any $\nu$-matching of $\cN(\T)$. 
Suppose that there exist facets $F_{3,a,b}\in U$ and 
$G_{3,c,d}\in N$ such that $|F_{3,a,b}\cap G_{3,c,d}|=1$ and 
$F_{3,a,b}$ is disjoint from every other facet of $N$. 
Then $a<c$, equivalently 
$F_{3,a,b}>_{\mathrm{lex}} G_{3,c,d}$.
\end{proposition}

\begin{proof}

Set $F_{3,a,b}=\N_{\T}[i_1]$ and $G_{3,c,d}=\N_{\T}[i_3]$ with 
$F_{3,a,b}\cap G_{3,c,d}=\{i_2\}$, where $i_1,i_2,i_3$ form a path in $\T$. 
Since $\N_{\T}[i_1]$ is disjoint from all other facets of $N$, Remark~\ref{rem=deg2}(2) 
implies that $\deg_{\T}(i_2)=2$. 
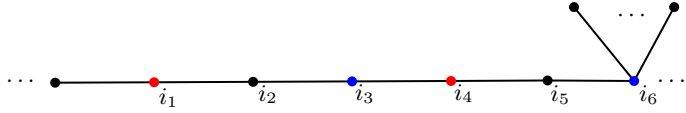
\begin{figure}[ht]
    \centering

\tikzset{every picture/.style={line width=0.75pt}} 

\begin{tikzpicture}[x=0.75pt,y=0.75pt,yscale=-1,xscale=1]

\draw    (181.03,1468.03) -- (230.43,1467.82) ;
\draw    (230.43,1467.82) -- (279.83,1467.6) ;
\draw    (279.83,1467.6) -- (329.23,1467.38) ;
\draw    (329.23,1467.38) -- (378.63,1467.17) ;
\draw    (426.82,1466.92) -- (378.63,1467.17) ;
\draw    (470,1467.17) -- (426.82,1466.92) ;

\draw    (470,1467.17) -- (440,1430) ;
\draw    (470,1467.17) -- (490,1430) ;

\draw (231.23,1469.42) node [anchor=north west][inner sep=0.75pt]  [font=\footnotesize]  {$i_{1}$};
\draw (280.83,1467.6) node [anchor=north west][inner sep=0.75pt]  [font=\footnotesize]  {$i_{2}$};
\draw (378.63,1467.17) node [anchor=north west][inner sep=0.75pt]  [font=\footnotesize]  {$i_{4}$};
\draw (329.23,1467.38) node [anchor=north west][inner sep=0.75pt]  [font=\footnotesize]  {$i_{3}$};
\draw (426.82,1467.92) node [anchor=north west][inner sep=0.75pt]  [font=\footnotesize]  {$i_{5}$};

\draw (471,1467.92) node [anchor=north west][inner sep=0.75pt]  [font=\footnotesize]  {$i_{6}$};

\draw (480,1465) node [anchor=north west][inner sep=0.75pt]  [font=\footnotesize]  {$\ldots$};

\draw (155,1465) node [anchor=north west][inner sep=0.75pt]  [font=\footnotesize]  {$\ldots$};


\draw (460,1430) node [anchor=north west][inner sep=0.75pt]  [font=\footnotesize]  {$\cdots$};

\filldraw[black](440,1430) circle (1.5pt) ;
  \filldraw[black] (490,1430) circle (1.5pt) ;

\filldraw[black](181.03,1468.03)  circle (1.5pt) ;
\filldraw[red] (230.43,1467.82) circle (1.5pt) ;
\filldraw[black] (279.83,1467.6) circle (1.5pt) ;
\filldraw[blue] (329.23,1467.38)  circle (1.5pt);
\filldraw[red] (378.63,1467.17) circle (1.5pt) ;
\filldraw[black] (426.82,1466.92) circle (1.5pt) ;
\filldraw[blue] (470,1467.17) circle (1.5pt) ;

\end{tikzpicture}
    \caption{The neighborhoods of the red vertices belong to $V_U$ and the the neighborhoods of the blue vertices belong to $V_N$.}
    \label{fig:placeholder}
\end{figure}

Following Notation~\ref{Notation.facet}, we have $\N_{\T}[i_2]=F_{3,j,k}$. Assume for contradiction that $c<a$. 
Then $c<j<a$, equivalently,
\[
\N_{\T}[i_3]>\N_{\T}[i_2]>\N_{\T}[i_1].
\]   

We first note that $i_3\in V_U$. Indeed, if $i_3\notin V_U$, then replacing $\N_{\T}[i_1]$ in $U$ by 
$\N_{\T}[i_2]$ yields a $\nu$-matching strictly larger than $U$, 
contradicting its maximality. Hence $i_3\in V_U$, and there exists 
a neighbor $i_4$ of $i_3$ with $\N_{\T}[i_4]\in U$ such that 
$\N_{\T}[i_2]\cap \N_{\T}[i_4]=\{i_3\}$. 

The vertex $i_4$ is not a leaf of $\T$, since otherwise 
$\N_{\T}[i_4]\subsetneq \N_{\T}[i_3]$, contradicting that 
$\N_{\T}[i_3]$ is a facet of $\mathcal{N}(\T)$. 
Moreover, by Proposition~\ref{prop:4face}, we have 
$|\N_{\T}[i_4]|\le 3$, and therefore $|\N_{\T}[i_4]|=3$. Furthermore, $\N_{\T}[i_4]$ must intersect at least two facets of $N$; 
otherwise replacing $\N_{\T}[i_3]$ in $N$ by 
$\N_{\T}[i_1]$ and $\N_{\T}[i_4]$ would produce a $(\nu+1)$-matching, 
a contradiction. Thus there exists a path $i_4,i_5,i_6$ in $\T$ such that 
$\N_{\T}[i_4]\cap \N_{\T}[i_6]=\{i_5\}$. 
In this situation the component cardinalities differ by one in the component containing $i_6$, so the hypotheses of 
Remark~\ref{rem=deg2} are satisfied. Hence $\deg_{\T}(i_5)=2$.

We next show that $i_6\in V_U$. 
If $i_6\notin V_U$, then replacing 
$\N_{\T}[i_1]$ and $\N_{\T}[i_4]$ in $U$ by 
$\N_{\T}[i_2]$ and $\N_{\T}[i_5]$ produces a 
$\nu$-matching strictly larger than $U$ with respect to 
$>_{\mathrm{lex}}$, contradicting the maximality of $U$. Hence $i_6\in V_U$, 
so there exists a neighbor $i_7$ of $i_6$ with 
$\N_{\T}[i_7]\in U$. As in the case of $i_4$, the vertex $i_7$ is not a leaf of $\T$, 
and we obtain a path $i_6,i_7,i_8$ in $\T$. 
Consider the matching
\[
M=(N\setminus\{\N_{\T}[i_3],\N_{\T}[i_6]\})
   \cup\{\N_{\T}[i_2],\N_{\T}[i_5]\}.
\]
In this situation the component cardinalities again differ by one 
in the component containing $i_7$, so the hypotheses of 
Remark~\ref{rem=deg2} are satisfied. 
Hence $\deg_{\T}(i_6)=2$.

Repeating the same argument alternately with facets of $N$ and $U$, we extend the path $i_1,i_2,i_3,i_4$, $i_5,i_6,i_7,\ldots$ strictly farther away from $i_1$ in $\T$. 
At each step the newly introduced vertices have degree $2$, while the terminal vertex lies in $V_U$. Since $\T$ is finite, the process must eventually terminate. 
At the final step the condition $i_{3k}\in V_U$ fails, 
and the complementary case yields a contradiction 
to the maximality of $U$. 
Hence $a<c$.
\end{proof}

\begin{lemma} \label{lem.ui}
    Let $\T$ be a tree. Let $U$ be the unique $\nu$-matching which is maximal with respect to $>_{\lex}$. Let $M$ and $N$ be two $\nu$-matchings with $F \in M\setminus U$ and $G \in N\setminus U$.  Assume that $(F\setminus G)\cap V_N=\emptyset$.  Then there exist $H\in U$ such that $H \cap F\neq \emptyset$ and $H \cap G\neq \emptyset$.
\end{lemma}
\begin{proof}
    We have $F\cap G \neq \emptyset$ because $M$ and $N$ are maximal matchings. Let $F=\N_{\T}[r]$ and $G=\N_{\T}[t]$. Suppose, to the contrary, that for every $H \in U$ either $H \cap \N_{\T}[r] = \emptyset$ or $H \cap \N_{\T}[t] = \emptyset$.

    Case I: Let $|\N_{\T}[r]\cap \N_{\T}[t]|=1$, and set  $\N_{\T}[r]\cap \N_{\T}[t] =\{s\}$. From our assumption we have that $\N_{\T}[x]\notin U$ for all $x\in \N_{\T}[s]$. Let $\T_1, \ldots, T_d$ be the connected components of $\T\setminus r$. Set $s\in \T_1$. By Lemma~\ref{lemma1.redu} and the assumption $(F\setminus G)\cap V_N=\emptyset$, it follows that $$|N_{\T_1}|=|M_{\T_1}|+1 \text{ and } |N_{\T_i}|=|M_{\T_i}|$$ for all $i=2,\ldots,d$. Note that if $|U_{\T_i}|=|M_{\T_i}|+1$ for some $2\leq i\leq d$, then  
    $$U_{\T_i}\cup N_{\T_1}\cup \bigcup_{k=2,k \neq i}^d M_{\T_k}$$ 
    forms a $(\nu+1)$-matching of $\cN(\T)$, since each $\T_i$ has pairwise disjoint set of vertices. This contradicts the matching number of $\cN(\T)$. Further, since $\N_{\T}[x]\notin U$ for all $x\in \N_{\T}[s]$,  by Lemma \ref{lemma1.redu}(i) we have $|U_{\T_1}|=|M_{\T_1}|$. Then from Lemma \ref{lemma1.redu}, we obtain $|U_{\T_i}|=|M_{\T_i}|$ for all $i=1,\ldots,d$ and $\N_{\T}[y]\in U$ for some $y\in \N_{\T}[r]\setminus \{r,s\}$. Then replacing the facets of $U_{\T_1}$ in $U$ by those of $N_{\T_1}$ produces a $(\nu+1)$-matching, a contradiction. This shows that $|\N_{\T}[r]\cap \N_{\T}[t]|=1$ does not hold. 
    
    Case II. Let $|\N_{\T}[r]\cap \N_{\T}[t]|=2$, where $\{r,t\}\in E(\T)$. Let $\T_1, \ldots, T_d$ be the connected components of $T\setminus r$. Set $t\in \T_1$. From our assumption we have that $\N_{\T}[x]\notin U$ for all $x\in V_{\N_{\T}[r],\N_{\T}[t]}$. In particular, $r\notin V_U$. By Lemma~\ref{lemma1.redu} and the assumption $(F\setminus G)\cap V_N=\emptyset$, it follows that $$|N_{\T_i}|=|M_{\T_i}|$$ for all $i=1,\ldots,d$. Note that if $|U_{\T_i}|=|M_{\T_i}|+1$ for some $2\leq i\leq d$ then  $$U_{\T_i}\cup (N\setminus N_{\T_i})$$ forms a $(\nu+1)$-matching of $\cN(\T)$, since $r\notin V_U$. This contradicts the matching number of $\cN(\T)$. Further,  since $\N_{\T}[x]\notin U$ for all $x\in \N_{\T}[t]$, by Lemma \ref{lemma1.redu}(i) we have $|U_{\T_1}|=|M_{\T_1}|$. Then from  Lemma \ref{lemma1.redu}, we obtain $|U_{\T_i}|=|M_{\T_i}|$ for all $i=1,\ldots,d$ and $r\in V_U$, a contradiction.  
\end{proof}

\begin{proposition} \label{lemma.H}
Let $\T$ be a tree satisfying \eqref{C1} and \eqref{C2}. Let $U$ be the unique $\nu$-matching maximal with respect to  $>_{\mathrm{lex}}$ (Notation~\ref{Notation.facet}), and let 
$M$ and $N$ be two $\nu$-matchings of $\mathcal{N}(\T)$ such that:
    \begin{enumerate} [label=(\roman*), ref=\roman*]
        \item \label{con.H1} $M>_\ell N$ (Notation~\ref{notation.total}),
    \item \label{con.H2} there do not exist facets $F\in M$ and $G\in N$ with 
    $F>G$ such that $F\setminus G$ is a singleton and 
    $F\setminus G \notin V_N$.
\end{enumerate}
  Then there exist facets $H\in U$ and $G\in N$, and a vertex 
$x\in H\setminus G$, such that $W=(N\setminus G) \cup H$ is a $\nu$-matching and $x\in V_M\setminus V_N$.
\end{proposition}

\begin{proof}
Let $M$ and $N$ be two $\nu$-matchings of $\mathcal N(\T)$ satisfying \eqref{con.H1} and \eqref{con.H2}. 
We first normalize $M$ as follows. Suppose there exist $F\in M\setminus U$ and $G\in N$ with $G>F$
such that $\bar M=(M\setminus F)\cup\{G\}$ is a $\nu$-matching. Then $\bar M>_\ell N$ still satisfies \eqref{con.H1} and \eqref{con.H2}.
Hence it suffices to prove the statement for $\bar M$ and $N$, since
$V_{\bar M}\setminus V_N\subseteq V_M\setminus V_N$. Consequently we may assume:

\medskip
\noindent{\em Assumption on $M$:}
For every $F\in M\setminus U$ there is no $G\in N$ with $G>F$ such that
$(M\setminus F)\cup\{G\}$ is a $\nu$-matching.

Note that this assumption implies
\begin{equation}\label{eq:1-dim}
\text{all one-dimensional facets of $M$ and $N$ coincide.}
\end{equation}
Indeed, let $F\in M$ with $|F|=2$. Then $F$ intersects some $G\in N$, since otherwise
$N\cup\{F\}$ would be a $(\nu+1)$-matching. If $F>G$, then $F$ and $G$
satisfy condition~\eqref{con.H2}, a contradiction. Hence $F<G$, and therefore $|G|=2$.
Moreover, $F\notin U$, since otherwise $(U\setminus\{F\})\cup\{G\}>_{\lex}U$,
contradicting the maximality of $U$. Since replacing $F$ by $G$ yields a $\nu$-matching,
the assumption forces $F=G$.

Let $M=\{F_1,\ldots,F_\nu\}$ and $N=\{G_1,\ldots,G_\nu\}$. By Remark~\ref{Rem.intersection} there exist $i,j\in[\nu]$
such that $F_i\cap G_j\neq\emptyset$, $F_i\neq G_j$, and
$F_i\cap G_k=\emptyset$ for all $k\neq j$.
After relabeling we may assume $i=j=1$. 

If $F_1\in U$, then $W=(N\setminus G_1)\cup F_1$ satisfies the conclusion,
so assume $F_1\notin U$. We may further assume that no facet of $M\cap U$ can be used directly.
Namely, 
\begin{equation}\label{eq:two-intersections}
\text{ for every $F\in M\cap U$, either } F\in N \text{ or } \bigl|\{G\in N:\; F\cap G\neq\emptyset\}\bigr|\ge 2.
\end{equation}
Indeed, if $F\in M\cap U$ intersects exactly one facet $G\in N$, then
$W=(N\setminus G)\cup F$ satisfies the conclusion.

We argue by decreasing induction on $n=|M\cap U|$.

\textit{Step 1: Let $n=\nu-1$.} Then all facets of $M$ except $F_1$ lie in $U$.
We first show $G_1\notin U$.
If $G_1\in U$, then $N'=(N\setminus G_1)\cup F_1$ is a $\nu$-matching and $M\neq N'$ because $M>_{\ell}N$.
Applying Remark~\ref{Rem.intersection} to $M$ and $N'$ yields
$F_i$ with $i\ne1$ intersecting exactly one facet of $N'$.
Since $F_i,G_1\in U$, we have $F_i\cap G_1=\emptyset$. Hence $F_i$ intersects exactly one facet of $N$, contradicting the assumption (\ref{eq:two-intersections}). Hence $G_1\notin U$.

By Lemma~\ref{lem.ui} there exists $H\in U$ such that
$H\cap F_1\neq\emptyset$ and $H\cap G_1\neq\emptyset$.
Since $F_1\in M\setminus U$, we have $H\notin M$,
and hence $U=(M\setminus F_1)\cup H$.
If $H$ intersected at least two facets of $N$, then by assumption (\ref{eq:two-intersections}), the bipartite
graph $B_{(U,N)}$ (Construction~\ref{cons:bipartite})
would not be a tree, contradicting Lemma~\ref{lem.tree}.
Hence $H$ intersects exactly one facet of $N$, namely $G_1$.
Thus $W=(N\setminus G_1)\cup H$ is a $\nu$-matching. It remains to show that $H\setminus G_1$ contains a vertex in $V_M\setminus V_N$.

If $|F_1\cap G_1|=1$, say $F_1\cap G_1=\{s\}$.
Since $|F_1|,|G_1|\ge3$ due to (\ref{eq:1-dim}), Remark~\ref{rem=deg2}(2) implies
$\deg_\T(s)=2$, so $H=\N_\T[s]$.
Hence $H\setminus G_1$ is a singleton and
$H\setminus G_1\subseteq F_1\setminus G_1$.
Because $F_1$ intersects no facet of $N$ other than $G_1$,
we have $F_1\setminus G_1\subseteq V_M\setminus V_N$,
which proves the claim.

It remains to show that $|F_1\cap G_1|=2$ cannot occur.
Suppose $|F_1\cap G_1|=2$. Then there exists a path $i_1,i_2,i_3,i_4$ in $\T$ such that $F_1=\N_\T[i_2]$ and $G_1=\N_\T[i_3]$, and either $H=\N_\T[i_1]$ or $H=\N_\T[i_4]$. First assume $H=\N_\T[i_1]$. Since $n=\nu-1$, $F$ does not intersect any facet of $U$ other than $H$, so we have $i_3\notin V_U$.
Because $\N_\T[i_1]\cap \N_\T[i_3]=\{i_2\}$ and
$H$ intersects no facet of $N$ other than $G_1=\N_\T[i_3]$,
Remark~\ref{rem=deg2}(2) applied to $N$ and $U$
yields $\deg_\T(i_2)=2$.
Since $i_3\notin V_U$, the maximality of $U$ implies
$\N_\T[i_1]>\N_\T[i_2]$, otherwise
$(U\setminus \N_\T[i_1])\cup \N_\T[i_2]>_{\lex}U$.
Consequently $|\N\T[i_1]|=3$, and level of $i_1$ is greater than level of $i_2$ in $\T$. Therefore
$\N_\T[i_2]>\N_\T[i_3]$, contradicting~\eqref{con.H2} and hence $H\neq\N_\T[i_1]$.
Then $H=\N_\T[i_4]$. By a symmetric argument we obtain
$\N_\T[i_2]<\N_\T[i_3]$. Since $M\setminus \N_\T[i_2]\neq N\setminus \N_\T[i_3]$
(otherwise $N>_{\ell} M$), applying Remark~\ref{Rem.intersection} to $M'=(M\setminus \N_\T[i_2])\cup{\N_\T[i_3]}$ and $N$ contradicts~\eqref{eq:two-intersections}.

Assume that the assertion holds for all $n\ge k+1$, where $k\le \nu-2$, and consider the case $n=k$.
Write
\[
M=\{F_1,\ldots,F_{\nu-k},F'_1,\ldots,F'_k\},
\]
where $F'_p\in U$ for $p=1,\ldots,k$ and $F_q\notin U$ for
$q=1,\ldots,\nu-k$.
Recall that, by Remark~\ref{Rem.intersection}, $F_1$ intersects exactly one facet of $N$, namely $G_1$. Hence $(N\setminus G_1)\cup F_1$ is a $\nu$-matching of
$\mathcal N(\T)$. Applying Remark~\ref{Rem.intersection} to $M$ and the resulting matching,
and repeating this procedure whenever possible, we obtain a
$\nu$-matching $N^\ast$ such that
\begin{enumerate}[label=(\alph*)]
\item $\bigl|\{F_1,\ldots,F_{\nu-k}\}\cap N^\ast\bigr|=\nu-k-1$;
\item for every $F'\in M\cap U$, either $F'\in N^\ast$ or
$\bigl|\{G\in N^\ast:\;F'\cap G\neq\emptyset\}\bigr|\ge 2$.
\end{enumerate}
Applying Remark~\ref{Rem.intersection} once more to $M$ and $N^\ast$,
and relabeling if necessary, we may assume that
$F_{\nu-k}\cap G_2\neq\emptyset$ and
$F_{\nu-k}$ does not intersect any other facet of $N^\ast$.

If $G_2\in U$, then $G_2$ does not intersect any facet of $M$
other than $F_{\nu-k}$. Hence $M'=(M\setminus F_{\nu-k})\cup G_2$ is a $\nu$-matching of $\mathcal N(\T)$ with $|M'\cap U|=k+1$.
Since $M'$ and $N$ satisfy conditions \eqref{con.H1} and \eqref{con.H2}, the induction
hypothesis yields the desired conclusion because $V_{M'}\setminus V_N\subseteq V_M\setminus V_N$.

Now, suppose that $G_2\notin U$.  By Lemma~\ref{lem.ui} there exists $H'\in U$ such that
$H'\cap F_{\nu-k}\neq\emptyset$ and $H'\cap G_2\neq\emptyset$. We distinguish the following two cases.

\textit{Case 1.} Suppose $|F_{\nu-k}\cap G_2|=1$.
Set $F_{\nu-k}=\N_\T[i_1]$, $G_2=\N_\T[i_3]$, and
$F_{\nu-k}\cap G_2=\{i_2\}$, where $i_1,i_2,i_3$ form a path in $\T$.
Since $|F_{\nu-k}|,|G_2|\ge 3$ by (\ref{eq:1-dim}), applying
Remark~\ref{rem=deg2}(2) to $M$ and $N^\ast$ yields
$\deg_\T(i_2)=2$. Hence $H'=\N_\T[i_2]=\{i_1,i_2,i_3\}\in U$, and it does not intersect any facet
in $M\cap U$. Moreover, as $i_1,i_2\in F_{\nu-k}$ and $i_3\in G_2$,
the facet $\N_\T[i_2]$ does not intersect any facet in $N^\ast\cap U$.
Thus $M''=(M\setminus F_{\nu-k})\cup H'$ is a $\nu$-matching with $|M''\cap U|=k+1$, $V_{M''}\setminus V_N\subseteq V_M\setminus V_N$. Moreover, $M''$ and $N$ satisfy conditions \eqref{con.H1} and \eqref{con.H2}.  Thus, the conclusion follows from the induction hypothesis on $n$.
    
\textit{Case 2.} Suppose $|F_{\nu-k}\cap G_2|=2$.
Set $F_{\nu-k}=\N_\T[i_2]$, $G_2=\N_\T[i_3]$, and choose
$i_1\in F_{\nu-k}\setminus G_2$ and $i_4\in G_2\setminus F_{\nu-k}$,
where $i_1,i_2,i_3,i_4$ form a path in $\T$. Then either $H'=\N_\T[i_1]$ or $H'=\N_\T[i_4]$.
Since $\T$ satisfies~\eqref{C2}, either $|\N_\T[i_2]|=3$ or
$|\N_\T[i_3]|=3$. We first note that $\N_\T[i_2]<\N_\T[i_3]$.
Otherwise $|\N_\T[i_2]|=3$ and $\N_\T[i_2]\setminus\N_\T[i_3]=\{i_1\}$ with $i_1\notin V_N$, contradicting \eqref{con.H2}. Hence $|\N_\T[i_3]|=3$.

Suppose that $\N_\T[i_4]\notin U$.
Then $H'=\N_\T[i_1]$.
Let $\T_1$ and $\T_2$ be the connected components of
$\T\setminus i_3$ with $i_2\in V(\T_1)$ and $i_4\in V(\T_2)$.
Since $\N_\T[i_2],\N_\T[i_3],\N_\T[i_4]\notin U$, we have $i_3\notin V_U$. We claim that $|U_{\T_1}|=|N_{\T_1}|+1$.
Let $\Sc_1,\ldots,\Sc_r$ be the connected components of $\T\setminus i_2$,
with $i_3\in\Sc_1$.
Since $i_3\notin V_U$, Lemma~\ref{lemma1.redu}(iii) gives
$|U_{\Sc_1}|=|M_{\Sc_1}|$.
Moreover, $i_2\in V_{N^\ast}$ and
$(\N_\T[i_2]\setminus\N_\T[i_3])\cap V_{N^\ast}=\emptyset$,
so Lemma~\ref{lemma1.redu}(ii) implies
$|N_{\Sc_1}|=|M_{\Sc_1}|$.
Hence $|U_{\Sc_1}|=|N_{\Sc_1}|$, which yields
$|U_{\T_2}|=|N_{\T_2}|$.
Applying Lemma~\ref{lemma1.redu}(iii) to the decomposition
$\T\setminus i_3$ now gives
$|U_{\T_1}|=|N_{\T_1}|+1$. By Remark~\ref{rem=deg2}(1) this gives $\deg_\T(i_2)=2$,
and hence $(U\setminus\N_\T[i_1])\cup\N_\T[i_2]>U,$ contradicting the maximality of $U$.
Thus, this case cannot occur.

Therefore $\N_\T[i_4]\in U$. Since $i_4\in G_2\in N^\ast$, it does not lie in $F_1,\ldots,F_{\nu-k-1}$, and because
$\N_\T[i_4]\cap F_{\nu-k}\neq\emptyset$,
we have $\N_\T[i_4]\notin M\cap U$.
Hence $i_4$ does not belong to any facet of $M\cap U$. Thus, $i_4\notin V_M$, and hence $\bar M=(M\setminus F_{\nu-k})\cup G_2 $ is a $\nu$-matching with $F_{\nu-k}>G_2$ and $F_{\nu-k}\notin U$, contradicting our assumption on $M$. This shows that this case cannot occur.
\end{proof}

We are now ready to prove the sufficient condition for the $\nu$-th squarefree power of the closed neighborhood ideal of a tree to be componentwise linear.

\begin{theorem}\label{thm.3}
Let $\T$ be a tree satisfying \eqref{C1} and \eqref{C2}, and set $J=NI(\T)$.
Then $J^{[\nu]}$ has linear quotients and is componentwise linear.
\end{theorem}

\begin{proof}
Let $U$ be the unique $\nu$-matching maximal with respect to $>_{\lex}$. 
Using Notation~\ref{notation.total}, we order the generators of $\G(J^{[\nu]})$ as 
$u_1>_\ell \cdots >_\ell u_n$. 
We prove that $J^{[\nu]}$ has linear quotients with respect to this order. Let $M$ and $N$ be the $\nu$-matchings corresponding to $u_a$ and $u_b$, respectively, with $M>_\ell N$. 
By \cite[Corollary 8.2.4]{HH2011}, it suffices to show that there exists a $\nu$-matching $W$ with 
$W>_\ell N$ such that $V_W\setminus V_N$ is a singleton and 
$V_W\setminus V_N\subseteq V_M\setminus V_N$.

Since $M>_\ell N$, we have $\beta(M)\ge\beta(N)$. 
If $\beta(M)>\beta(N)$, then there exist facets $F\in M$ and $G\in N$ with $|F|=2$ and $|G|\ge3$ such that $F$ intersects only $G$ among the facets of $N$. 
Hence $W=(N\setminus G)\cup F$ is a $\nu$-matching. 
Since $|F|=2$, we have $\beta(W)>\beta(N)$ and thus $W>_\ell N$. 
Moreover $V_W\setminus V_N$ is a singleton and 
$V_W\setminus V_N\subseteq V_M\setminus V_N$, as required.

We may therefore assume $\beta(M)=\beta(N)$. 
If there exist facets $F\in M$ and $G\in N$ with $F>G$ such that 
$F\setminus G$ is a singleton and $F\setminus G\notin V_N$, that is, $F$ intersects no facet of $N$ other than $G$. 
Thus $W=(N\setminus G)\cup F$ is a $\nu$-matching, and since $F>G$ we obtain $W>_\ell N$. 
As $F\setminus G$ is a singleton, the required condition holds.

Hence we may assume that no such pair of facets exists. 
By Proposition~\ref{lemma.H}, there exist facets $H\in U$ and $G\in N$, and a vertex $x\in H\setminus G$, such that 
$W=(N\setminus G)\cup H$ is a $\nu$-matching and $x\in V_M\setminus V_N$. 
In particular, $\level_U(W)=\level_U(N)-1$, and therefore $W>_\ell N$. 
There are two possibilities for $H$: either $H\in M$ or $H\notin M$.

\textit{Case A.} Suppose that $H\notin M$. 
As shown in the proof of Proposition~\ref{lemma.H}, the set $H\setminus G$ is a singleton and $W$ satisfies the required conditions.

\textit{Case B.} Suppose that $H\in M$. 
We distinguish two subcases according to $|H\cap G|$.

\textit{Case B.I.} Assume $|H \cap G|=2$. 
By Proposition~\ref{prop:4face}, $|H|\le 3$, and $|H|=2$ is impossible since $H$ and $G$ are distinct facets of $\cN(\T)$. 
Hence $|H|=3$. 
Thus $|H\setminus G|=1$, so $V_W\setminus V_N$ is a singleton. 
Since $H\in M$, we have $V_W\setminus V_N \subseteq V_M\setminus V_N$.

\textit{Case B.II.} Assume that $|H \cap G|=1$. 
By Proposition~\ref{prop:4face}, we have $|H|\le 3$. 
If $|H|=2$, then $H\setminus G$ is a singleton and $H\setminus G\notin V_N$, 
contradicting our assumption. Hence $|H|=3$. 
If $|G|=2$, then $(U\setminus H)\cup G>_{\lex}U$, contradicting the maximality of $U$. 
Thus $|G|\ge 3$.

Set $H=\N_\T[i_1]$ and $G=\N_\T[i_3]$ with 
$H\cap G=\{i_2\}$, where $i_1,i_2,i_3$ form a path in $\T$. 
Since $W$ and $N$ differ in one facet, Condition~\eqref{C1} implies 
$\deg_\T(i_2)=2$. Consider
\[
W'=(N\setminus \N_\T[i_3])\cup \N_\T[i_2].
\]
Since $H\setminus G\notin V_N$, we obtain $\N_\T[i_1]\cap V_{N\setminus \N_\T[i_3]}=\emptyset$. Also, 
$\N_\T[i_2]\subset V_{\{\N_\T[i_1],\N_\T[i_3]\}}$ gives
$\N_\T[i_2]\cap V_{N\setminus \N_\T[i_3]}=\emptyset$. 
Thus $W'$ is a $\nu$-matching of $\cN(\T)$, and clearly 
$\level(W')=\level(N)$.

We claim that $W'>_{\lex}N$. 
Once this is established, then $W'>_{\ell} N$, and the fact that
$|\N_\T[i_2]\setminus \N_\T[i_3]|=1$ implies that 
$V_{W'}\setminus V_N$ is a singleton. Since 
$\N_\T[i_2]\setminus \N_\T[i_3]\subseteq 
\N_\T[i_1]\setminus \N_\T[i_3]\in V_M$, we obtain $V_{W'}\setminus V_N \subseteq V_W\setminus V_N 
\subseteq V_M\setminus V_N,$ as required.

Now we prove the claim. Write $\N_\T[i_2]=F_{i,j,k}$ and $\N_\T[i_3]=F_{i',j',k'}$
(see Notation~\ref{Notation.facet}). Since 
$i=|\N_\T[i_2]|=3$, we have $i\le i'$. 
If $i<i'$, then $\N_\T[i_2]>\N_\T[i_3]$, and hence 
$W'>_{\lex}N$, as required. Now suppose $i=i'=3$. By Proposition~\ref{prop:level}, the level of $i_1$ in $\T$ is strictly less than that of $i_3$. Since $i_2$ lies on the path 
between $i_1$ and $i_3$, we obtain $j<j'$, and hence $W'>_{\lex}N$, as required. 
\end{proof}

We conclude this section by presenting a class of graphs whose highest non-vanishing squarefree power of the closed neighborhood ideal is componentwise linear. 

\begin{corollary}
    If $G$ is a path graph or a whiskered path graph (that is, a graph obtained by attaching a whicker to every vertex of a path graph), then the ideal $NI(G)^{[\nu]}$ is componentwise linear. 
\end{corollary}

\section{Regularity of the highest squarefree power} \label{sec.regularity}

In this section, we demonstrate that if either Condition~\eqref{C1} or \eqref{C2} is not satisfied, then one can construct a tree $G$ for which 
$\reg(NI(G)^{[\nu]})$ is arbitrarily larger than $\deg(NI(G)^{[\nu]})$. This provides a positive answer to Question \ref{qus.regm}. We begin by defining a class of trees that does not satisfy Condition \eqref{C2}.

\begin{example} \label{ex.reg}
Let $G_1$ be the graph defined in Figure \ref{fig:graphG} with $n \geq 1$. 
It is easy to verify that the matching number of $\cN({G_1})$ is $2(n+1)+1$. Moreover, $$M_1=\bigcup_{i=1}^{n+1}\{\N_{G_1}[r''_i],\N_{G_1}[s''_i]\} \cup \{\N_{G_1}[r]\}$$ and $$M_2=\bigcup_{i=1}^{n+1}\{\N_{G_1}[r''_i],\N_{G_1}[s''_i]\} \cup \{\N_{G_1}[s]\}$$ are the only $2n+3$-matchings of $\cN({G_1})$. 
Furthermore, $G_1$ satisfies Condition \eqref{C1}, but does not satisfy Condition \eqref{C2}.

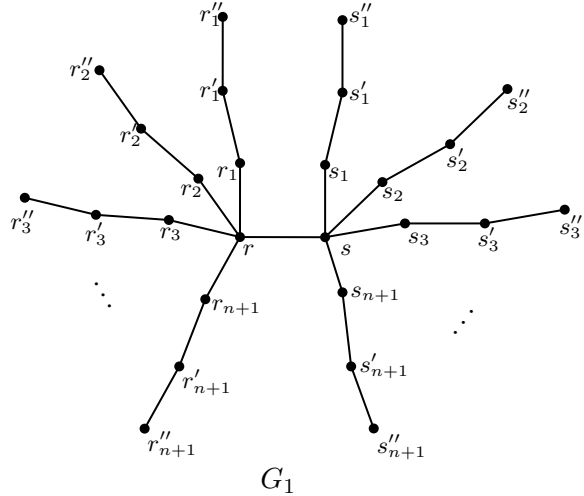
\begin{figure}[ht]
    \centering
    \tikzset{every picture/.style={line width=0.75pt}} 

\begin{tikzpicture}[x=0.65pt,y=0.65pt,yscale=-1,xscale=1]

\draw    (200.03,829.43) -- (200,872.33) ;
\draw    (176.03,838.43) -- (200,872.33) ;
\draw    (159.03,862.43) -- (200,872.33) ;
\draw    (180.03,908.43) -- (200,872.33) ;
\draw    (200,872.33) -- (249.03,872.43) ;
\draw    (249.03,872.43) -- (249.03,830.43) ;
\draw    (249.03,872.43) -- (282.03,840.43) ;
\draw    (249.03,872.43) -- (295.03,864.43) ;
\draw    (249.03,872.43) -- (259.03,904.43) ;
\draw    (190.03,787.43) -- (200.03,829.43) ;
\draw    (190.07,744.53) -- (190.03,787.43) ;
\draw    (143.03,809.43) -- (176.03,838.43) ;
\draw    (119.07,775.53) -- (143.03,809.43) ;
\draw    (117.03,859.43) -- (159.03,862.43) ;
\draw    (76.07,849.53) -- (117.03,859.43) ;
\draw    (165.03,947.43) -- (180.03,908.43) ;
\draw    (145.07,983.53) -- (165.03,947.43) ;
\draw    (249.03,830.43) -- (259.03,788.43) ;
\draw    (259.03,788.43) -- (259.03,746.43) ;
\draw    (282.03,840.43) -- (321.03,818.43) ;
\draw    (321.03,818.43) -- (354.03,786.43) ;
\draw    (295.03,864.43) -- (341.03,864.43) ;
\draw    (341.03,864.43) -- (387.03,856.43) ;
\draw    (259.03,904.43) -- (264.03,947.43) ;
\draw    (264.03,947.43) -- (277.03,983.43) ;

\draw (200,875) node [anchor=north west][inner sep=0.75pt]  [font=\footnotesize]  {$r$};
\draw (256,875) node [anchor=north west][inner sep=0.75pt]  [font=\footnotesize]  {$s$};
\draw (185,830) node [anchor=north west][inner sep=0.75pt]  [font=\footnotesize]  {$r_{1}$};
\draw (164.03,840) node [anchor=north west][inner sep=0.75pt]  [font=\footnotesize]  {$r_{2}$};
\draw (152.03,865) node [anchor=north west][inner sep=0.75pt]  [font=\footnotesize]  {$r_{3}$};
\draw (182,907) node [anchor=north west][inner sep=0.75pt]  [font=\footnotesize]  {$r_{n+1}$};
\draw (249.03,830.43) node [anchor=north west][inner sep=0.75pt]  [font=\footnotesize]  {$s_{1}$};
\draw (280,842) node [anchor=north west][inner sep=0.75pt]  [font=\footnotesize]  {$s_{2}$};
\draw (295,868) node [anchor=north west][inner sep=0.75pt]  [font=\footnotesize]  {$s_{3}$};
\draw (262,899.43) node [anchor=north west][inner sep=0.75pt]  [font=\footnotesize]  {$s_{n+1}$};
\draw (175.03,778.43) node [anchor=north west][inner sep=0.75pt]  [font=\footnotesize]  {$r'_{1}$};
\draw (175,734.43) node [anchor=north west][inner sep=0.75pt]  [font=\footnotesize]  {$r''_{1}$};
\draw (129.03,804.43) node [anchor=north west][inner sep=0.75pt]  [font=\footnotesize]  {$r'_{2}$};
\draw (100,766) node [anchor=north west][inner sep=0.75pt]  [font=\footnotesize]  {$r''_{2}$};
\draw (109.03,862) node [anchor=north west][inner sep=0.75pt]  [font=\footnotesize]  {$r'_{3}$};
\draw (66.03,855) node [anchor=north west][inner sep=0.75pt]  [font=\footnotesize]  {$r''_{3}$};
\draw (165,947.43) node [anchor=north west][inner sep=0.75pt]  [font=\footnotesize]  {$r'_{n+1}$};
\draw (145.07,983.53) node [anchor=north west][inner sep=0.75pt]  [font=\footnotesize]  {$r''_{n+1}$};
\draw (261.03,780) node [anchor=north west][inner sep=0.75pt]  [font=\footnotesize]  {$s'_{1}$};
\draw (261,735) node [anchor=north west][inner sep=0.75pt]  [font=\footnotesize]  {$s''_{1}$};
\draw (317.03,819) node [anchor=north west][inner sep=0.75pt]  [font=\footnotesize]  {$s'_{2}$};
\draw (352.03,785) node [anchor=north west][inner sep=0.75pt]  [font=\footnotesize]  {$s''_{2}$};
\draw (335.03,865) node [anchor=north west][inner sep=0.75pt]  [font=\footnotesize]  {$s'_{3}$};
\draw (382.03,858) node [anchor=north west][inner sep=0.75pt]  [font=\footnotesize]  {$s''_{3}$};
\draw (267.03,938) node [anchor=north west][inner sep=0.75pt]  [font=\footnotesize]  {$s'_{n+1}$};
\draw (277.03,985) node [anchor=north west][inner sep=0.75pt]  [font=\footnotesize]  {$s''_{n+1}$};
\draw (119.68,893.51) node [anchor=north west][inner sep=0.75pt]  [rotate=-59.76,xslant=0]  {$\cdots $};
\draw (317.73,927.8) node [anchor=north west][inner sep=0.75pt]  [rotate=-307.35,xslant=0]  {$\cdots $};

\draw (210,1005) node [anchor=north west][inner sep=0.75pt]    {$G_1$};

\filldraw[black] (200.03,829.43)  circle (1.5pt) ;
\filldraw[black] (200,872.33) circle (1.5pt) ;
\filldraw[black] (176.03,838.43) circle (1.5pt) ;
\filldraw[black] (159.03,862.43)  circle (1.5pt) ;
\filldraw[black] (180.03,908.43) circle (1.5pt) ;
\filldraw[black] (249.03,872.43) circle (1.5pt) ;\filldraw[black] (249.03,830.43)  circle (1.5pt) ;
\filldraw[black] (282.03,840.43) circle (1.5pt) ;
\filldraw[black] (295.03,864.43) circle (1.5pt) ;
\filldraw[black] (259.03,904.43) circle (1.5pt) ;
\filldraw[black] (190.03,787.43) circle (1.5pt) ;
\filldraw[black] (190.07,744.53) circle (1.5pt) ;
\filldraw[black] (143.03,809.43) circle (1.5pt) ;
\filldraw[black] (119.07,775.53) circle (1.5pt) ;
\filldraw[black] (117.03,859.43) circle (1.5pt) ;
\filldraw[black] (76.07,849.53) circle (1.5pt) ;
\filldraw[black] (165.03,947.43) circle (1.5pt) ;
\filldraw[black] (145.07,983.53) circle (1.5pt) ;
\filldraw[black] (259.03,788.43) circle (1.5pt) ;
\filldraw[black] (259.03,746.43) circle (1.5pt) ;
\filldraw[black] (321.03,818.43) circle (1.5pt) ;
\filldraw[black] (354.03,786.43) circle (1.5pt) ;
\filldraw[black] (341.03,864.43) circle (1.5pt) ;
\filldraw[black] (387.03,856.43) circle (1.5pt) ;
\filldraw[black] (264.03,947.43) circle (1.5pt) ;
\filldraw[black] (277.03,983.43) circle (1.5pt) ;
\end{tikzpicture}
    \caption{The graph $G_1$.}
    \label{fig:graphG}
\end{figure}    
\end{example}

\begin{theorem} \label{thm.regm}
    Let $m$ be a positive integer. Then there exists a tree $G$ that satisfies Condition \eqref{C1} but fails to satisfy Condition \eqref{C2}, such that $\reg(NI(G)^{[\nu]})-\deg(NI(G)^{[\nu]})=m$.
\end{theorem}
\begin{proof}
    Fix a positive integer $m$. Let $G$ be a graph defined in Example \ref{ex.reg} with $n=m$. Then it follows that $NI(G)^{[\nu]}=({\bf x}_{V_{M_1}},{\bf x}_{V_{M_2}})$, where $M_1$ and $M_2$ are defined in Example \ref{ex.reg}. It is clear that $\deg(NI(G)^{[\nu]})=5(m+1)+2$. Consider the following short exact sequence 
    \begin{equation} \label{eq:SES2}
    0 \longrightarrow \frac{S}{(({\bf x}_{V_{M_1}}):{\bf x}_{V_{M_2}})}(-(5(m+1)+2)) \longrightarrow \frac{S}{({\bf x}_{V_{M_1}})}
  \longrightarrow \frac{S}{NI(G)^{[\nu]}}   \longrightarrow  0.
\end{equation}
Observe that $(({\bf x}_{V_{M_1}}):{\bf x}_{V_{M_2}})=({\bf x}_{A})$, where $A=\{r_1,\ldots, r_{m+1}\}$. It is easy to see that $\reg(S/({\bf x}_{V_{M_1}}))=5(m+1)+1$ and $\reg(S/(({\bf x}_{V_{M_2}}):{\bf x}_{V_{M_1}}))=m$. Then, applying Lemma \ref{Lemma1.Reg} to Equation (\ref{eq:SES2}) yields $\reg(S/{NI(G)^{[\nu]}})=5(m+1)+m+1$. Hence, $\reg(NI(G)^{[\nu]})-\deg(NI(G)^{[\nu]})=m$, as desired.
\end{proof}

In the following, we define the class of trees that does not satisfy Condition \eqref{C1}. 

\begin{example} \label{ex.reg2}
    Let $G_2$ be the graph defined in Figure \ref{fig:graphG2} with $n \geq 1$. It is easy to see that the matching number of $\cN(G_2)$ is $4n+1$. Moreover, 
    \begin{equation*}
    \begin{split}
        M_1 &=\bigcup_{i=1}^n\{\N_{G_2}[r''_i],\N_{G_2}[\tilde{\tilde{r}}_i],\N_{G_2}[s''_i],\N_{G_2}[\tilde{\tilde{s}}_i]\} \cup \{\N_{G_2}[t_1]\}, \\
        M_2 &=\bigcup_{i=1}^n\{\N_{G_2}[r''_i],\N_{G_2}[\tilde{\tilde{r}}_i],\N_{G_2}[s''_i],\N_{G_2}[\tilde{\tilde{s}}_i]\} \cup \{\N_{G_2}[r]\}, \text{ and } \\
        M_3 &=\bigcup_{i=1}^n\{\N_{G_2}[r''_i],\N_{G_2}[\tilde{\tilde{r}}_i],\N_{G_2}[s''_i],\N_{G_2}[\tilde{\tilde{s}}_i]\} \cup \{\N_{G_2}[s]\},
    \end{split}
    \end{equation*}
    are the $4n+1$-matchings of $\cN(G_2)$. Furthermore, $G_2$ satisfies Condition \eqref{C2}, whereas it fails to satisfy Condition \eqref{C1}.
    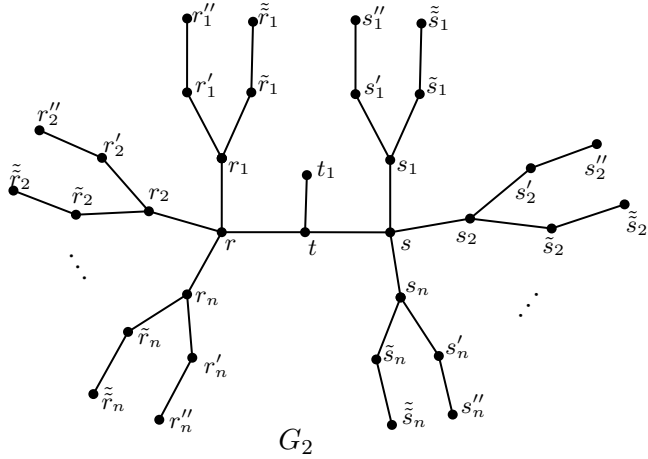
\begin{figure}[ht]
        \centering
\tikzset{every picture/.style={line width=0.75pt}} 

\begin{tikzpicture}[x=0.65pt,y=0.65pt,yscale=-1,xscale=1]

\draw    (239.03,1204.27) -- (239,1247.17) ;
\draw    (197.03,1235.03) -- (239,1247.17) ;
\draw    (219.03,1283.27) -- (239,1247.17) ;
\draw    (287,1247.17) -- (336.03,1247.27) ;
\draw    (336.03,1247.27) -- (336.03,1205.27) ;
\draw    (336.03,1247.27) -- (382.03,1239.27) ;
\draw    (336.03,1247.27) -- (342.03,1285.03) ;
\draw    (219.03,1166.03) -- (239.03,1204.27) ;
\draw    (219.07,1123.13) -- (219.03,1166.03) ;
\draw    (155.03,1237.03) -- (197.03,1235.03) ;
\draw    (119.03,1223.03) -- (155.03,1237.03) ;
\draw    (185.03,1305.03) -- (219.03,1283.27) ;
\draw    (165.07,1341.13) -- (185.03,1305.03) ;
\draw    (382.03,1239.27) -- (417.03,1210.03) ;
\draw    (417.03,1210.03) -- (455.03,1196.03) ;
\draw    (382.03,1239.27) -- (429.03,1245.03) ;
\draw    (429.03,1245.03) -- (471.03,1231.03) ;
\draw    (342.03,1285.03) -- (364.03,1319.03) ;
\draw    (328.03,1321.03) -- (337.03,1359.03) ;
\draw    (239,1247.17) -- (287,1247.17) ;
\draw    (287,1247.17) -- (288.03,1214.03) ;
\draw    (239.03,1204.27) -- (256.03,1166.03) ;
\draw    (257.03,1125.03) -- (256.03,1166.03) ;
\draw    (170.03,1204.03) -- (197.03,1235.03) ;
\draw    (134.03,1188.03) -- (170.03,1204.03) ;
\draw    (316.03,1167.03) -- (336.03,1205.27) ;
\draw    (316.07,1124.13) -- (316.03,1167.03) ;
\draw    (336.03,1205.27) -- (353.03,1167.03) ;
\draw    (354.03,1126.03) -- (353.03,1167.03) ;
\draw    (222.03,1320.03) -- (219.03,1283.27) ;
\draw    (203.03,1356.03) -- (222.03,1320.03) ;
\draw    (342.03,1285.03) -- (328.03,1321.03) ;
\draw    (364.03,1319.03) -- (372.03,1353.03) ;

\draw (239,1250) node [anchor=north west][inner sep=0.75pt]  [font=\footnotesize]  {$r$};
\draw (340,1250) node [anchor=north west][inner sep=0.75pt]  [font=\footnotesize]  {$s$};
\draw (241,1204.27) node [anchor=north west][inner sep=0.75pt]  [font=\footnotesize]  {$r_{1}$};

\draw (222,1280) node [anchor=north west][inner sep=0.75pt]  [font=\footnotesize]  {$r_{n}$};
\draw (338,1204) node [anchor=north west][inner sep=0.75pt]  [font=\footnotesize]  {$s_{1}$};
\draw (372,1245) node [anchor=north west][inner sep=0.75pt]  [font=\footnotesize]  {$s_{2}$};
\draw (343,1274.27) node [anchor=north west][inner sep=0.75pt]  [font=\footnotesize]  {$s_{n}$};
\draw (222.03,1153.27) node [anchor=north west][inner sep=0.75pt]  [font=\footnotesize]  {$r'_{1}$};
\draw (220.03,1112.27) node [anchor=north west][inner sep=0.75pt]  [font=\footnotesize]  {$r''_{1}$};
\draw (195.03,1220) node [anchor=north west][inner sep=0.75pt]  [font=\footnotesize]  {$r_{2}$};
\draw (169.03,1188) node [anchor=north west][inner sep=0.75pt]  [font=\footnotesize]  {$r'_{2}$};
\draw (131.03,1170) node [anchor=north west][inner sep=0.75pt]  [font=\footnotesize]  {$r''_{2}$};
\draw (151.03,1219) node [anchor=north west][inner sep=0.75pt]  [font=\footnotesize]  {$\tilde{r}_{2}$};
\draw (115.03,1206) node [anchor=north west][inner sep=0.75pt]  [font=\footnotesize]  {$\tilde{\tilde{r}}_{2}$};

\draw (227,1316) node [anchor=north west][inner sep=0.75pt]  [font=\footnotesize]  {$r'_{n}$};
\draw (207,1348) node [anchor=north west][inner sep=0.75pt]  [font=\footnotesize]  {$r''_{n}$};
\draw (189,1301) node [anchor=north west][inner sep=0.75pt]  [font=\footnotesize]  {$\tilde{r}_{n}$};
\draw (168,1335.37) node [anchor=north west][inner sep=0.75pt]  [font=\footnotesize]  {$\tilde{\tilde{r}}_{n}$};

\draw (406,1215) node [anchor=north west][inner sep=0.75pt]  [font=\footnotesize]  {$s'_{2}$};
\draw (446.03,1201) node [anchor=north west][inner sep=0.75pt]  [font=\footnotesize]  {$s''_{2}$};
\draw (422.03,1248) node [anchor=north west][inner sep=0.75pt]  [font=\footnotesize]  {$\tilde{s}_{2}$};
\draw (470.03,1233) node [anchor=north west][inner sep=0.75pt]  [font=\footnotesize]  {$\tilde{\tilde{s}}_{2}$};

\draw (366.03,1305.27) node [anchor=north west][inner sep=0.75pt]  [font=\footnotesize]  {$s'_{n}$};
\draw (375.03,1339.27) node [anchor=north west][inner sep=0.75pt]  [font=\footnotesize]  {$s''_{n}$};
\draw (330,1311) node [anchor=north west][inner sep=0.75pt]  [font=\footnotesize]  {$\tilde{s}_{n}$};
\draw (341.03,1344.27) node [anchor=north west][inner sep=0.75pt]  [font=\footnotesize]  {$\tilde{\tilde{s}}_{n}$};

\draw (154.68,1255) node [anchor=north west][inner sep=0.75pt]  [rotate=-59.76,xslant=0]  {$\cdots $};
\draw (404.73,1295.64) node [anchor=north west][inner sep=0.75pt]  [rotate=-307.35,xslant=0]  {$\cdots $};
\draw (287,1250) node [anchor=north west][inner sep=0.75pt]  [font=\footnotesize]  {$t$};
\draw (292,1204.17) node [anchor=north west][inner sep=0.75pt]  [font=\footnotesize]  {$t_{1}$};
\draw (258.03,1154.27) node [anchor=north west][inner sep=0.75pt]  [font=\footnotesize]  {$\tilde{r}_{1}$};
\draw (258.03,1112.27) node [anchor=north west][inner sep=0.75pt]  [font=\footnotesize]  {$\tilde{\tilde{r}}    _{1}$};
\draw (319.03,1154.27) node [anchor=north west][inner sep=0.75pt]  [font=\footnotesize]  {$s'_{1}$};
\draw (318,1114) node [anchor=north west][inner sep=0.75pt]  [font=\footnotesize]  {$s''_{1}$};
\draw (355.03,1155.27) node [anchor=north west][inner sep=0.75pt]  [font=\footnotesize]  {$\tilde{s}_{1}$};
\draw (356,1114) node [anchor=north west][inner sep=0.75pt]  [font=\footnotesize]  {$\tilde{\tilde{s}}_{1}$};
\draw (270,1360) node [anchor=north west][inner sep=0.75pt]   {$G_{2}$};

\filldraw[black] (239.03,1204.27)  circle (1.5pt) ;
\filldraw[black] (197.03,1235.03) circle (1.5pt) ;
\filldraw[black] (219.03,1283.27) circle (1.5pt) ;
\filldraw[black] (336.03,1247.27)  circle (1.5pt) ;
\filldraw[black] (287,1247.17) circle (1.5pt) ;
\filldraw[black] (239,1247.17) circle (1.5pt) ;
\filldraw[black] (336.03,1205.27)  circle (1.5pt) ;
\filldraw[black] (382.03,1239.27) circle (1.5pt) ;
\filldraw[black] (342.03,1285.03) circle (1.5pt) ;
\filldraw[black] (219.03,1166.03) circle (1.5pt) ;
\filldraw[black] (239.03,1204.27) circle (1.5pt) ;
\filldraw[black] (219.07,1123.13) circle (1.5pt) ;
\filldraw[black] (155.03,1237.03) circle (1.5pt) ;
\filldraw[black] (119.03,1223.03) circle (1.5pt) ;
\filldraw[black] (185.03,1305.03) circle (1.5pt) ;
\filldraw[black] (165.07,1341.13) circle (1.5pt) ;
\filldraw[black] (417.03,1210.03) circle (1.5pt) ;
\filldraw[black] (455.03,1196.03) circle (1.5pt) ;
\filldraw[black] (429.03,1245.03) circle (1.5pt) ;
\filldraw[black] (364.03,1319.03) circle (1.5pt) ;
\filldraw[black] (471.03,1231.03) circle (1.5pt) ;
\filldraw[black] (328.03,1321.03) circle (1.5pt) ;
\filldraw[black] (337.03,1359.03) circle (1.5pt) ;
\filldraw[black] (288.03,1214.03) circle (1.5pt) ;
\filldraw[black] (239.03,1204.27) circle (1.5pt) ;
\filldraw[black] (257.03,1125.03) circle (1.5pt) ;
\filldraw[black] (256.03,1166.03) circle (1.5pt) ;
\filldraw[black] (170.03,1204.03) circle (1.5pt) ;
\filldraw[black] (134.03,1188.03) circle (1.5pt) ;
\filldraw[black] (316.07,1124.13) circle (1.5pt) ;
\filldraw[black] (316.03,1167.03) circle (1.5pt) ;
\filldraw[black] (353.03,1167.03) circle (1.5pt) ;
\filldraw[black] (354.03,1126.03) circle (1.5pt) ;
\filldraw[black] (222.03,1320.03) circle (1.5pt) ;
\filldraw[black] (203.03,1356.03) circle (1.5pt) ;
\filldraw[black] (372.03,1353.03) circle (1.5pt) ;
\filldraw[black] (342.03,1285.03) circle (1.5pt) ;
\end{tikzpicture}
        \caption{The graph $G_2$}
        \label{fig:graphG2}
    \end{figure}
\end{example}

\begin{theorem} \label{thm.regm2}
    Let $m$ be a positive integer. Then there exists a tree $G$ that satisfies Condition \eqref{C2} but fails to satisfy Condition \eqref{C1}, such that $\reg(NI(G)^{[\nu]})-\deg(NI(G)^{[\nu]})=m$.
\end{theorem}
\begin{proof}
    Let $m$ be a positive integer. Let $G$ be a graph defined in Example \ref{ex.reg2} with $n=m$. Then it follows that $NI(G)^{[\nu]}=({\bf x}_{V_{M_1}},{\bf x}_{V_{M_2}},{\bf x}_{V_{M_3}})$, where $M_1$, $M_2$,  and $M_3$ are defined in Example \ref{ex.reg2}. Since $(({\bf x}_{V_{M_1}}):{\bf x}_{V_{M_2}})=(x_{t_1})$, it follows that the ideal $({\bf x}_{V_{M_1}},{\bf x}_{V_{M_2}})$ has linear quotients. Therefore, $\reg(S/({\bf x}_{V_{M_1}},{\bf x}_{V_{M_2}}))=9m+1$. Consider the following short exact sequence 
    \begin{equation} \label{eq:SES1b}
    0 \longrightarrow \frac{S}{(({\bf x}_{V_{M_1}},{\bf x}_{V_{M_2}}):{\bf x}_{V_{M_3}})}(-(9m+2)) \longrightarrow \frac{S}{({\bf x}_{V_{M_1}},{\bf x}_{V_{M_2}})} \longrightarrow \frac{S}{NI(G)^{[\nu]}}   \longrightarrow  0.
    \end{equation}
    Observe that $(({\bf x}_{V_{M_1}},{\bf x}_{V_{M_2}}):{\bf x}_{V_{M_3}})=(x_{t_1},{\bf x}_{A})$, where $A=\{r,r_1,\ldots, r_{m}\}$. Then, it follows that $\reg(S/({\bf x}_{V_{M_1}},{\bf x}_{V_{M_2}}):{\bf x}_{V_{M_3}}))=m$. Then, applying Lemma \ref{Lemma1.Reg} to Equation (\ref{eq:SES1b}) yields $\reg(S/{NI(G)^{[\nu]}})=10m+1$. Since $\deg(NI(G)^{[\nu]}) = 9m + 2$, we conclude that $\reg(NI(G)^{[\nu]}) - \deg(NI(G)^{[\nu]}) = m$. This completes the proof.
\end{proof}

Next, we study the regularity of the $\nu$-th square-free power of the closed neighborhood ideal of caterpillar graphs. We begin by recalling the definition of a caterpillar graph.

\vspace{2mm}

A graph $G$ is said to be a caterpillar graph if there exists a path $P \subseteq G$ such that every vertex of $G$ either lies on $P$ or is adjacent to a vertex on $P$. The path $P \subseteq G$ is referred to as the central path of $G$.

\begin{theorem} \label{thm.reg}
    Let $G$ be a caterpillar graph. Then, one has 
$$\reg \left(\frac{S}{NI(G)^{[\nu]}} \right)=\deg(NI(G)^{[\nu]})-1.$$
\end{theorem}

To complete the proof of Theorem \ref{thm.reg}, we need to prove a  couple of lemmas. We begin by fixing some notation. Let $I$ be an ideal of $S$ and let $\G(I)=\{u_1,\ldots, u_m\}$ be the minimal generating set. We set $I_j=(u_1,\ldots,u_j)$ for $j=1,2,\ldots,m$.

\begin{lemma} \label{lem.colon}
    Let $G$ be a caterpillar graph. Let $I$ be the $\nu$-th squarefree power of the closed neighborhood ideal of $G$ with $\G(I)=\{u_1,\ldots, u_m\}$. Then there exists a total order on $\G(I)$ such that for all $j=2,\ldots,m$, the colon ideal $(I_{j-1}:u_{j})$ is generated either by variables or by variables together with distinct quadratic monomials.
\end{lemma}

\begin{proof}
Following Notation~\ref{Notation.facet}, we order the generators of $I$ by $u_1 >_{\lex} \cdots >_{\lex} u_m$.
For each $i$, let $M_i$ denote the $\nu$-matching associated with $u_i$.
Let
\[
M_a=\{F_{i_1,j_1,k_1},\ldots,F_{i_\nu,j_\nu,k_\nu}\}
\quad\text{and}\quad
M_b=\{F_{i'_1,j'_1,k'_1},\ldots,F_{i'_\nu,j'_\nu,k'_\nu}\}
\]
be two $\nu$-matchings of $\cN(G)$ written in decreasing order with respect to $>$. 
Assume that $a<b$ and set $s=\min\{\ell : F_{i_\ell,j_\ell,k_\ell}\neq F_{i'_\ell,j'_\ell,k'_\ell}\}$.
Since $M_a>_{\lex}M_b$, we have
$F_{i_s,j_s,k_s} > F_{i'_s,j'_s,k'_s}$. 
We claim that $F_{i_s,j_s,k_s}$ intersects exactly one facet of $M_b$.
First observe that there exists $t\ge s$ such that $F_{i_s,j_s,k_s}\cap F_{i'_t,j'_t,k'_t}\neq\emptyset$,
since otherwise $M_b\cup\{F_{i_s,j_s,k_s}\}$ would form a $(\nu+1)$-matching of $\cN(G)$, contradicting the matching number. Note that every facet of $\cN(G)$ has cardinality at most $3$.  If $|F_{i_s,j_s,k_s}|=3$, then for all $q\ge s$ we have
\begin{equation}\label{eq.s>q}
F_{i_s,j_s,k_s}>F_{i'_q,j'_q,k'_q},
\end{equation}
and hence $|F_{i'_q,j'_q,k'_q}|=3$ for all $q\ge s$. 
If $F_{i_s,j_s,k_s}$ intersected two distinct facets of $M_b$, say 
$F_{i'_p,j'_p,k'_p}$ and $F_{i'_q,j'_q,k'_q}$ with $p,q\ge s$. 
Since these facets are contained in the central path $P$, it follows that either 
$F_{i'_p,j'_p,k'_p}>F_{i_s,j_s,k_s}$ or 
$F_{i'_q,j'_q,k'_q}>F_{i_s,j_s,k_s}$, contradicting \eqref{eq.s>q}. 
Thus $F_{i_s,j_s,k_s}$ intersects exactly one facet of $M_b$. If $|F_{i_s,j_s,k_s}|=2$, the conclusion clearly holds. 
Let $t$ be the unique index such that
$
F_{i_s,j_s,k_s}\cap F_{i'_t,j'_t,k'_t}\neq\emptyset$.
Then
\[
M_c=(\{F_{i_s,j_s,k_s}\}\cup M_b)\setminus\{F_{i'_t,j'_t,k'_t}\}
\]
is a $\nu$-matching of $\cN(G)$, and since 
$F_{i_s,j_s,k_s}>F_{i'_t,j'_t,k'_t}$ we obtain $M_c>_{\lex}M_b$.

Hence, for every $a<b$ there exists $c<b$ such that $M_c$ and $M_b$ differ by exactly one facet. 
Since every facet of $\cN(G)$ has size at most $3$, we have$
|V_{M_c}\setminus V_{M_b}|\le 2$.
This implies that the colon ideal $(I_{j-1}:u_j)$ is generated either by variables or by variables together with quadratic monomials.

It remains to show that the quadratic generators are distinct. 
Suppose that $x_{i_1}x_{i_2}$ and $x_{i_3}x_{i_4}$ belong to $\G(I_{j-1}:u_j)$ with $
|\{i_1,i_2\}\cap\{i_3,i_4\}|=1$.
Then $\{i_1,i_2\}$ and $\{i_3,i_4\}$ are edges of $G$. 
Without loss of generality assume $i_2=i_3$, so that $
P'=\{\{i_1,i_2\},\{i_2,i_4\}\}$
forms a path in $G$, and every vertex of $P'$ lies outside $V_{M_j}$.

If $\deg_G(i_2)=2$, then $M_j\cup \N_G[i_2]$ is a matching of $\cN(G)$, contradicting the matching number. 
If $\deg_G(i_2)\ge 3$, there exists $i_5\in\N_G[i_2]\setminus\{i_1,i_4\}$ and again 
$M_j\cup\N_G[i_5]$ forms a matching, yielding the same contradiction. 
Therefore $\{i_1,i_2\}\cap\{i_3,i_4\}=\emptyset$,
which proves that the quadratic generators are distinct.
\end{proof}

\begin{lemma} \label{lem.degree}
    Under the assumption of Lemma \ref{lem.colon}, let $j \in [m]$. If the colon ideal $(I_{j-1}:u_{j})$ is generated by variables together with $k$ distinct quadratic monomials, then there exists $i>j$ such that $\deg(u_{j}) +k-1 < \deg(u_i)$.
\end{lemma}

\begin{proof}
Let $u_1>_{\lex}\cdots>_{\lex}u_m$ be the order on $\G(I)$ introduced in the proof of 
Lemma~\ref{lem.colon}, and let 
$M_i=\{F_{i_1},\ldots,F_{i_\nu}\}$ denote the $\nu$-matching associated with $u_i$.

Assume that there exists $j\in[m]$ such that the minimal generating set of the
colon ideal $(I_{j-1}:u_j)$ contains exactly two distinct quadratic monomials
$x_{i_1}x_{i_2}$ and $x_{i_3}x_{i_4}$ with the level of $i_1$ strictly less than that of $i_3$. 
Since $x_{i_1}x_{i_2}\in\G(I_{j-1}:u_j)$, it follows from the proof of 
Lemma~\ref{lem.colon} that there exists $r<j$ such that $M_r$ and $M_j$
differ by exactly one facet and
$(u_r:u_j)=x_{i_1}x_{i_2}$.
Hence there exist indices $p$ and $q$ such that
\[
M_r\setminus\{F_{r_p}\}=M_j\setminus\{F_{j_q}\}
\quad\text{and}\quad
F_{r_p}\setminus F_{j_q}=\{i_1,i_2\}.
\]

Moreover, by minimality of the colon generator there is no facet $F\in\mathcal{F}(\cN(G))$ such that $F\setminus F_{j_q}\subset F_{r_p}\setminus F_{j_q}$.
Consequently, along the central path 
\[
P=\{\ldots,i_{\theta-2},i_{\theta-1},i_\theta,i_{\theta+1},i_{\theta+2},\ldots\}
\]
of the caterpillar $G$ we have $F_{r_p}=\N_G[i_{\theta-1}]$ and $F_{j_q}=\N_G[i_{\theta+1}]$, with $\deg_G(i_\theta)\ge3$.  Note that $\deg_G(i_{\theta-2}) = 2$ and $\deg_G(i_{\theta-3}) \geq 2$, otherwise, the facet corresponding to a pendant vertex is disjoint from $V_{M_j}$, which contradicts the matching number of $\cN(G)$. There exists
a facet $F_{j_{l'}}\in M_j$ such that $
\N_G[i_{\theta-2}]\cap F_{j_{l'}}\neq\emptyset$;
otherwise, $M_j\cup\{\N_G[i_{\theta-2}]\}$ would form a matching of $\cN(G)$, contradicting the matching number.

Set
\[
N_1=(M_j\setminus\{F_{j_{l'}}\})\cup\N_G[i_{\theta-2}].
\]
If $|F_{j_{l'}}|=2$, we set $M_{i'}=N_1$.
Otherwise $|F_{j_{l'}}|=3$, in which case $F_{j_{l'}}=\N_G[i_{\theta-4}]$. 
Repeating the above argument along the path, we construct successively
\[
N_t=(N_{t-1}\setminus\{F_{j_l}\})\cup\N_G[i_{\theta-(3t-1)}],
\]
where $F_{j_l}$ denotes the corresponding facet in $M_j$.

This process terminates at some step $t$ when the replaced facet has
cardinality $2$. Indeed, if all replaced facets had cardinality $3$ and
$i_{\theta-3(t+1)}$ were a pendant vertex of $G$, then $
N_t\cup\{\N_G[i_{\theta-3(t+1)}]\}$
would form a matching of $\cN(G)$, again contradicting the matching number.
Let $M_{i'}=N_t$. By construction, $
|\{F\in M_j:|F|=2\}|>|\{F\in M_{i'}:|F|=2\}|$,
and hence $M_j>_\lex M_{i'}$. Moreover, $
|V_{M_{i'}}|=|V_{M_j}|+1$.

Let $
A=\{\N_G[i_{\theta-(3k-1)}]:k=1,\ldots,t\}$.
Since $V_A\setminus\{i_1,i_2\}\subset V_{M_j}$ while
$\{i_3,i_4\}\nsubseteq V_{M_j}$, the vertices $i_3$ and $i_4$
do not belong to $V_A$. Consequently, the colon ideal
$(I_{i'-1}:u_{i'})$ contains the quadratic monomial $x_{i_3}x_{i_4}$, since the level of $i_1$ is less than that of $i_3$. Repeating the same argument, we obtain a matching $M_{i''}$ such that
\[
M_{i'}>_\lex M_{i''} \qquad\text{and}\qquad |V_{M_{i''}}|=|V_{M_{i'}}|+1.
\]
Therefore, $
\deg(u_{i''})=\deg(u_j)+2$,
which completes the proof.
\end{proof}

\begin{proof}[Proof of Theorem~\ref{thm.reg}]  
    Let $I$ be the $\nu$-th squarefree power of the closed neighborhood ideal of $G$ and let $\G(I)=\{u_1,\ldots, u_m\}$.  It is clear that $\reg(S/I) \geq \max_{1\leq i \leq n}\{\deg(u_i)\}-1$. Thus it suffices to prove that $\reg(S/I) \leq \max_{1\leq i \leq n}\{\deg(u_i)\}-1$. Using Notation \ref{Notation.facet}, we order the elements of $\G(I)$ by $u_1>_\lex \cdots >_\lex u_m$. Set $d_j=\deg(u_j)$ for $j =1,\ldots,m$. For $j =2,\ldots,m$ consider the following sequence of short exact sequences  

\begin{equation} \label{eq:SES}
    0 \longrightarrow \frac{S}{(I_{j-1}:u_j)}(-d_j) \longrightarrow \frac{S}{I_{j-1}}
  \longrightarrow \frac{S}{I_j}   \longrightarrow  0.
\end{equation}
Applying Lemma \ref{Lemma1.Reg} to Equation (\ref{eq:SES}), yields $\reg(S/(I))\leq \max\{\alpha,\reg(S/I_1)\}$, where $\alpha = \max_{2\leq j \leq m}\{\reg(S/(I_{j-1}:u_j))+d_j-1\}$. By \cite[Theorem 2.18]{Z2004}, if $\Delta$ is a $1$-dimensional forest, then $\reg(S/I(\Delta))$ equals the maximal number of pairwise disjoint facets in $\Delta$. Hence, by Lemma~\ref{lem.colon}, the ideal $(I_{j-1}:u_j)$ is generated by variables and $\beta_j$ distinct quadratic monomials, where $\beta_j$ denotes the number of quadratic generators. Consequently,
$\reg(S/(I_{j-1}:u_j))=\beta_j$. 
Thus $\alpha = \max_{2\leq j \leq m}\{\beta_j+d_j-1\}$. By Lemma~\ref{lem.degree}, for every $j$ there exists $i>j$ such that $\beta_j+d_j \leq d_i$. Therefore, $\alpha \leq \max_{2\leq j \leq m}\{d_j-1\}$. Since $\reg(S/I_1)=d_1-1$, we conclude that
$
\reg(S/I)\le 
\max\{\max_{2\le j\le m}(d_j-1),\,d_1-1\}
=\max_{1\le j\le m}\{d_j-1\}, \text{ as desired}.
$
\end{proof}

\end{document}